\documentclass{amsart}

\usepackage[utf8]{inputenc}
\usepackage[T1]{fontenc}

\usepackage{amsmath}
\usepackage{amsthm}
\usepackage{amssymb}
\usepackage{amsfonts}
\usepackage{dsfont}
\usepackage{url}
\usepackage{graphicx}
\usepackage{hyperref}
\usepackage{color}

\theoremstyle{plain}
\newtheorem{theorem}{Theorem}[section]
\newtheorem{lemma}[theorem]{Lemma}
\newtheorem{proposition}[theorem]{Proposition}
\newtheorem{corollary}[theorem]{Corollary}
\newtheorem{remark}[theorem]{Remark}
\newtheorem{definition}[theorem]{Definition}\newtheorem{as}{Assumption}

\renewcommand{\P}{\mathbb{P}} 
\newcommand{\E}{\mathbb{E}} 
\newcommand{\N}{\mathbb{N}} 
\newcommand{\R}{\mathbb{R}} 
\newcommand{\bR}{ (-\infty, \infty]} 
\newcommand{\1}{\mathds{1}} 

\newcommand{\ud}{{\rm d}}
\newcommand{\const}{{\rm C}}
\newcommand{\Mcal}{\mathcal{M}}
\newcommand{\Ccal}{\mathcal{C}}

\newcommand{\shift}{\mathfrak{T}}
\newcommand{\scale}{\mathfrak{S}}
\newcommand{\tip}{\mathfrak{V}}

\begin{document}

\title[Extremes of BRW with stretched exponential displacements]{The extremal point process for branching random walk with stretched exponential displacements}
\author{Piotr Dyszewski}
	\address{Instytut Matematyczny Uniwersytetu Wroc\l{}awskiego, Pl. Grunwaldzki 2/4 50-384, 
   Wroc\l{}aw, Poland }
\email{piotr.dyszewski@math.uni.wroc.pl}

\author{Nina Gantert}
\address{Fakultät für Mathematik, Technische Universität München, 
  Boltzmannstr.~3, 85748 Garching, Germany}
\email{gantert@ma.tum.de}
	\thanks{The research of PD was partially supported by the National Science Centre, Poland (Sonata, grant number 2020/39/D/ST1/00258) }
\maketitle

\begin{abstract}
	We investigate a branching random walk where the~displacements are independent from the branching mechanism and have a stretched exponential 
	distribution.  
	We describe the positions of the particles in the vicinity of the rightmost particle in~terms of point process convergence.
	As a consequence we give a~new limit theorem for the position of the rightmost particle. Our methods rely on providing precise large deviations for 
	sums of i.i.d. random variables with stretched exponential distributions outside the so-called one big jump regime.\smallskip

	\noindent \textbf{Keywords.} branching random walk, limit theorem, point processes, stretched exponential distribution, precise large deviations

	\noindent \textbf{AMS 2000 subject classification:} 60F10, 60J80, 60G50. 
\end{abstract}

\section{Introduction} \label{sec:intro}

	Branching random walk (BRW) has been studied for more than forty years and during this time, due to its relation with the KPP equation, a 
	lot of effort was made to understand the behaviour of its extremes. In the existing literature one can distinguish two different 
	regimes depending on the laws of the 
	displacements, however not many results address BRW outside of this two classical cases.  
	The first one treats displacements with thin tails, i.e. possessing some exponential moments. 
	In this case the asymptotic behaviour of the extremes of BRW is determined by the contribution coming from many, typical particles, 
	see~\cite{J.M.Hammersley1974, 75:kingman:first, 76:biggins:first, 13:aidekon:convergence, madule2017}. 
	In the second regime one investigates regularly varying steps, in this case the asymptotic behaviour of the extremes of BRW is determined by a few, 
	atypically big displacements, see~\cite{Durrett1983, bhattacharya:2017:point}. 
	The contrast between these two cases raises the question about a regime in which both big jumps and typical displacements 
	contribute to the asymptotics 
	of the extremes of BRW. 
	It turned out that in the case of~stretched exponential displacements one can see a clear transition between the two aforementioned 
	regimes~\cite{dyszewski:max:2022}. We aim to give a complete picture of the extremes of BRW in this case.\medskip

	In what follows we study BRW which is a branching process with a spatial component that can be described as follows. 
	The process will evolve at time 
	epochs $n \in \N = \{ 0,1,2, \ldots \}$. 
	Suppose that at time $n=0$ one particle is placed at the origin of the real line $\R$. At time $n=1$ the initial particle splits into a 
	random number 
	of new particles which are randomly 
	displaced from their place of birth. We will always assume that all displacements are independent copies of a given random 
	variable $X$. From now on, 
	all particles independently evolve in the same way as the initial particle. Denote by 
	$\Lambda_n$ the point process obtained by putting a~unit mass in the position of each particle present in the system at time $n$. 
	Note that $\Lambda_n$ is not
	necessarily simple, since two particles can occupy the same position. 
	The sequence $\{\Lambda_n\}_{n \in \N}$ is called a branching random walk. 
	We refer to Section~\ref{sec:prel} for a more detailed set-up and to the recent monograph \cite{Shi2015} for an introduction to the topic.\medskip
	
	For a point measure $\eta = \sum_i\delta_{x_i}$ and $a, b \in \R$, $a>0$ denote the translation of $\eta$ and scaling of $\eta$ via
	\begin{equation*}
		\shift_b \eta = \sum_{i} \delta_{x_i+b} \qquad \mbox{and} \qquad \scale_a \eta = \sum_{i} \delta_{a x_i} 
	\end{equation*}
	respectively.
	The aim of the present article is a description of the asymptotic behaviour of the extremes of $\{\Lambda_n\}_{n \in \N}$ by finding sequences 
	$\{a_n\}_{n \in \N}$ and $\{b_n\}_{n \in \N}$ for which $\scale_{a_n^{-1}}\shift_{-b_n} \Lambda_n$
	has a non-trivial limit in distribution in the sense of point process convergence, which we will make precise in Section~\ref{sec:prel}. 
	Such a~choice is not unique by any means. We will pick  the pair 
	$a_n$ and $b_n$ which allows to describe the position of the rightmost particle. In other words, we will work with sequences $a_n$ and $b_n$ 
	for which 
	the limiting configuration of $\scale_{a_n^{-1}}\shift_{-b_n} \Lambda_n$ does 
	not contain a sequence diverging to infinity. Under some mild moment conditions on the number of offspring of the~initial particle, 
	the asymptotic behaviour of 
	$\{\Lambda_n\}_{n \in \N}$ is determined mainly 
	by the law of the displacements, i.e. the law of $X$. In the classical case one works assuming Cram\'er's condition
	\begin{equation}\label{eq:1:cramer}
		\E \left[ e^{s|X|} \right] <\infty \quad \mbox{for some  $s>0$}.
	\end{equation}
 	Under condition~\eqref{eq:1:cramer} the contribution of a single displacement is negligible and the position of the rightmost particle 
	is concentrated 
	around its mean which moves at a linear speed. 
	One can show that there are 
	constants $\const_1, \const_2$ such that 
	\begin{equation*}
		\shift_{-\const_1 n - \const_2 \log n}\Lambda_n 
	\end{equation*}
	converges to a non-trivial limit, see~\cite{madule2017}. If the condition~\eqref{eq:1:cramer} is not satisfied, some additional 
	scaling is necessary. 
	For example, consider the case of 
	regularly varying tails, i.e. a random variable $X$ such that as $t$ tends to infinity
	\begin{equation}\label{eq:1:regular}
		\P\left[ X > t\right] \sim \const \cdot t^{-\gamma},
	\end{equation}
	for some constants $\gamma, \const >0$. Here and throughout the article the above relation means that the quotient of the two 
	quantities tends to one 
	(see the end of this Section for a~precise statement). 
	The regularly varying distributions follow the so called "one big jump principle" (see \eqref{eq:2:onebig}). Using this fact one can observe a 
	different limiting behaviour of $\Lambda_n$. 
	Namely, if one denotes by $m>1$ the expected 
	number of children of the initial particle, one can show that
	\begin{equation*}
		\scale_{m^{-n/\gamma}} \Lambda_n
	\end{equation*}
	converges in law to a non-trivial limit, see~\cite{bhattacharya:2017:point}. The scaling present above is necessary to 
	compensate the contribution of atypical, 
	huge displacements which, in contrast to 
	the previous case, is not negligible. \medskip
	
	We will investigate an intermediate case between~\eqref{eq:1:cramer} and~\eqref{eq:1:regular}, namely stretched exponential, 
	or Weibull, displacements 
	with tails of the form
	\begin{equation*}
		\P\left[X>t \right] \sim  e^{- t^r \ell(t)},	
	\end{equation*}
	for $r \in (0,1)$ and a slowly varying function $\ell$. Recall that a function $\ell: \R \to (0, \infty)$ is slowly varying (at infinity) 
	if for any 
	constant $\const>0$, $\ell(\const t) \sim \ell(t)$ as $t \to \infty$. 
	We will show in our main result that
	\begin{equation*}
		 \scale_{n^{1-1/r}\ell_2(n)}\shift_{-n^{1/r}\ell_1(n)}\Lambda_n 
	\end{equation*}
	converges for an appropriate choice of slowly varying functions $\ell_1$ and $\ell_2$.  Moreover we will provide a detailed asymptotic expansion of 
	$n^{1/r}\ell_1(n)$.
	As a consequence one gets a~new limit theorem for $M_n$, the position of the rightmost particle in generation $n$.
	Namely, we will infer that
	\begin{equation*}
		n^{1-1/r}\ell_2(n)\left( M_n - n^{1/r}\ell_1(n) \right)
	\end{equation*}
	converges in distribution to a random shift of the Gumbel law.
	Such results were obtained recently for constant $\ell$ in the case $r \leq  2/3$, see~\cite{dyszewski:max:2022}. 
	Our proof relies on precise large deviation estimates for sums of i.i.d. random variables with stretched exponential tails. Since the corresponding 
	results established in the 
	sixties~\cite{1969:nagaev:integral, 1969:nagaev:integral2} were too 
	implicit for our needs, we revisit the problem and provide a more explicit description of the large deviation probabilities on the 
	polynomial scale.\medskip
	
	The paper is organized as follows. In Section~\ref{sec:prel} we give a precise description of the model and our assumptions.
	Our main results, Theorem \ref{thm:2:main} and Corollary \ref{cor:main}, are presented in  
	Section~\ref{sec:main}. 
	In Sections~\ref{sec:trimming} and~\ref{sec:dec} we present the main steps of our arguments, namely trimming and 
	decoupling of the random weighted tree. 
	Lastly, in Section~\ref{sec:lim} we give the final arguments of our proof and in Section~\ref{sec:LD} we prove precise large deviations for sums of 
	stretched exponential random variables, 
	which may be of independent interest.\medskip
	
	 Throughout the article we write $g(t) \sim f(t)$ for two functions $f, g \colon \R \to \R$ whenever
	 \begin{equation*}
	 	\lim_{t \to \infty} f(t)/ g(t) =1
	 \end{equation*}  
	 and we write $f(t) \gg g(t)$ if $g(t) = o(f(t))$, i.e. if $f(t)/g(t) \to 0$.
	 We will denote by $\const$ positive constants whose value is not important for us. 
	 Note that the value of $"\const"$ may change from line to line. 
	
\section{Preliminaries}\label{sec:prel}	

	Let us now give a precise definition of our model. Since we will assume that the displacements and the reproduction mechanism are independent, 
	we will introduce both of these components separately. 
	If we denote by $Z_n = \Lambda_n(\R)$ the number of particles present in the system at time $n$, then the sequence 
	$Z=\{Z_n\}_{n \in \N}$ forms a Galton-
	Watson process starting from 
	$Z_0=1$, i.e. with one particle, and a~reproduction law with mean
	\begin{equation*}
		m = \E \left[Z_1\right].
	\end{equation*} 
	It is well known from the classical literature, see~\cite[Theorem I.A.5.1]{athreya:1972:branching}, that if $\P[Z_1=1]$ is strictly less than $1$, 
	$Z$ survives with positive probability, that is
	\begin{equation*}
		 \P\left[ \forall \: n \geq 0, \: Z_n >0 \right]>0
	\end{equation*}
	if and only if $m>1$. Since our objects become trivial on the set of extinction of the underlying Galton-Watson process, one can work under the 
	conditional probability
	\begin{equation*}
		\P^*\left[ \: \cdot \: \right] = \P\left[ \: \cdot \: |   \:   \forall \: n \in \N, \:Z_n >0\: \right].
	\end{equation*}
	In fact one can recast our main results in terms of weak convergence under the probability measure $\P^*$. We will however, for notational simplicity,
	continue to work under $\P$.
	Under a mild integrability assumption, known as the Kesten-Stigum condition given below, the exponential growth rate of $Z$ is strictly 
	positive on the event of survival (see~\eqref{Wdef} below).

	\begin{as} \label{as:BP} 
		We have $\P[Z_1=1]<1$ and $\E[ Z_1 \log^+ Z_1] <\infty$ and the Galton-Watson process $Z$ is supercritical, that is $m>1$. 
	\end{as}

	Note that by the martingale convergence theorem there is a random variable $W$ such that $\P$- almost surely, 
	\begin{equation}\label{Wdef}
		\lim_{n \to \infty}m^{-n}Z_n = W.
	\end{equation}
	If Assumption~\ref{as:BP} is satisfied then the convergence holds also in $L^1$ and moreover $\P^*[W > 0] =1$, 
	see for example~\cite[Chapter 2]{Shi2015}. This means that 
	\begin{equation*}
		\P[W=0] = \P[\exists n\in \N, \: Z_n=0].
	\end{equation*}
	The random variable $W$ will play a~crucial role in our results: it will be a~shift parameter in the limit of $\Lambda_n$. \medskip

	We will now describe the branching structure in more detail and introduce the displacements. Denote by 
	$\mathbb{T} \subseteq \mathbb{U} = \{\emptyset\} \cup \bigcup_{n \geq 1}\N^n$ the Galton-Watson tree corresponding to 
	$Z$ with Ulam-Harris labelling. That is, label the initial particle with $\emptyset$ and its children with
	$\{1,2, \ldots , Z_1\} \subset \N$ and for each particle with a label $w \in \mathbb{T}$, label its children with labels of the form 
	$(w,1), (w,2), \ldots , (w, N(w))$, where $N(w)$ is the random variable denoting the number of children of 
	the individual labelled with $w$. Write $|w|=n$ if $w$ is a label of a particle from generation $w$, in other words if $w \in \N^n$. 
	Denote by $(\emptyset, w]$ the unique vertex path from $\emptyset$ to $w$ 
	(excluding $\emptyset$ and including $w$) and write $ w \wedge z$ to denote 
	the last common ancestor of $w$ and $z$, i.e. the unique element of $\mathbb{T}$ for which 
	$[\emptyset, w] \cap [\emptyset, z] = [\emptyset, w \wedge z]$. 
	Finally, write $w \leq z$ whenever $w$ is an ancestor of 
	$z$, that is  $w \in [\emptyset, z]$ and $\mathbb{T}_n$ for the particles from the $n$th generation present in the system, i.e. 
	$\mathbb{T}_n =\{ w \in \mathbb{T} \: : \: |w|=n\}$. \medskip
	
	To model the displacements consider a family $\{ X_w \: : \: w \in  \mathbb{U} \}$ of independent copies of $X$ which are independent of $Z$. 
	The random variable $X_w$ represents the displacement of the particle with label $w$
	from its place of birth. The position 
	of the particle with label $w$ is therefore represented via
	\begin{equation*}
		V(w) = \sum_{ z \in (\emptyset, w]} X_z.
	\end{equation*} 
	To study the collection $\{ V(w)\}_{w \in \mathbb{T}_n}$ define a sequence of point processes by
	\begin{equation*}
		\Lambda_n = \sum_{|w|=n} \delta_{V(w)}.
	\end{equation*}
	We call the sequence $\{\Lambda_n\}_{n \in \N}$ a branching random walk. We will regard $\Lambda_n$ as a~point process, that is  a random element of 
	$\Mcal_p=\Mcal_p(\bR)$, the space of point measures on $\bR$ equipped with the topology 
	of vague convergence. Here $(-\infty, +\infty]$ denotes the homomorphic image of $(0,1]$ where closed neighbourhoods of $+\infty$ are compact.  
	Recall that $\eta$ is a point measure, i.e. an element of $\Mcal_p$, if it can be written in the form 
	$\eta = \sum_i \delta_{x_i}$ for some sequence $\{x_i\}_i \subset \bR$ without an accumulation point. 
	We refer to~\cite[Chapter 3.1]{resnick:2013:extreme} for a more 
	in-depth description of $\Mcal_p$.
	We equip $\Mcal_p$ with the topology of vague convergence, and say 
	that $\eta_n \to \eta$ in $\Mcal_p$ (or vaguely) if for any function $f$ from the set $\Ccal_K^+=\Ccal_{K}^+(\bR)$ of nonnegative, 
	continuous functions 
	with compact 
	support on $\bR$, $\int f(x) \eta_n(\ud x) \to \int f(x) \eta(\ud x)$.
	In our results we use weak convergence of random elements of
	$\Mcal_p$ (denoted by $\Rightarrow$), which is equivalent to the following condition, see \cite[Proposition 3.19]{resnick:2013:extreme}.
	For point processes $\Theta_n$, $\Theta$ we have 
		\begin{equation*}
			\Theta_n \Rightarrow \Theta \quad \mbox{ in }\Mcal_p,
		\end{equation*}
		if and only if for any $f\in \Ccal_K^+$,  
		\begin{equation*}
			 \int f\left( x \right) \:  \Theta_n (\ud x)  \overset{d}\to \int f(x)  \: \Theta(\ud x),
		\end{equation*}
		where $\overset{d}\to$ denotes the convergence in law of random variables.
	Our goal is to study weak convergence of 
	\begin{equation*}
		\scale_{a_n^{-1}}\shift_{-b_n}\Lambda_n=\sum_{|w|=n} \delta_{  a_n^{-1}(V(w) - b_n)}
	\end{equation*}	
	in $\Mcal_p$. Aiming to find a correct choice of $b_n$'s and $a_n$'s we first note that the convergence of 
	$\scale_{a_n^{-1}}\shift_{-b_n}\Lambda_n$ in $\Mcal_p$ describes the behaviour of $\Lambda_n$ 
	near its right frontier, that is the location of the rightmost particle, which we can write as
	\begin{equation}\label{maxdef}
		M_n = \max_{|w|=n} V(w)
	\end{equation}
	with the convention that the maximum of an empty set is $-\infty$. 
	The first order of magnitude of $M_n$ can be predicted using the first moment method.
 	Take $X_1, X_2, \ldots$ to be i.i.d. copies of $X$ and let $S_n = \sum_{k=1}^nX_k$ denote the corresponding random walk.
	For a~sequence of real numbers $\{x_n\}_{n\in \N}$ write
	\begin{equation}\label{eq:2:firstmoment}
		\P\left[ M_n >x_n \right]  = \P \left[ \exists w \in \mathbb{T}_n, \:  V(w) >x_n \right] 
		\leq \E[Z_n] \P[S_n>x_n]= m^n\P[S_n>x_n]\, . 
	\end{equation}
	For the estimate~\eqref{eq:2:firstmoment} to give any non-trivial information about the asymptotics of $M_n$,  one needs to choose the sequence 
	$\{x_n\}_{n\in \N}$ such that the probabilities $\P[S_n> x_n]$ decay as $m^{-n}$, that is
	exponentially fast. Here large deviation estimates for the random walk $\{S_n\}_{n \in \N}$ come into play.\medskip
	
	In the classical case, under Cram\'er's condition~\eqref{eq:1:cramer}, it is well known that if one takes the sequence of the form 
	$x_n = \rho n$ for 
	$\rho > \E[X]$, by the Bahadur-Ranga-Rao theorem, 
	see~\cite[Theorem 3.7.4]{dembo:1998:large},
	\begin{equation}\label{eq:2:rao}
		\P[S_n > \rho n] \sim \const(\rho) n^{-1/2} e^{-\psi^*(\rho) n},
	\end{equation}
	where $\psi^*(\rho) = \sup_{s} \left[ s\rho - \psi(s) \right]$, with $\psi(t) = \log \E\left[ e^{tX}\right]$ and a constant $\const(\rho)$ 
	depending on $\rho$. 
	Provided that $\rho < \sup \{ \psi'(s) \: : \: \psi(s)<\infty\}$, under some mild regularity assumptions, there exists 
	$s_\rho$ such that $\psi'(s_\rho) = \rho$ and then $\psi^*(\rho) = \rho s_\rho - \psi(s_\rho)$.
	Now, our estimate for $\P[M_n>x_n]$ suggests that one should find a constant $\rho$ such that
	$\psi^*(\rho) = \rho s_{\rho} -\psi(s_{\rho}) = \log m$.
	Then one could expect that, after establishing the lower bound corresponding to~\eqref{eq:2:firstmoment}, 
	\begin{equation*}
		\lim_{n \to \infty} M_n/n = \const_1 = \rho  \qquad \P^*\text{- a.s. }
	\end{equation*}
	This is indeed true and dates back to~\cite{J.M.Hammersley1974, 76:biggins:first,  75:kingman:first}. 
	Since the seventies a lot of effort was devoted to a 
	precise description of the asymptotics of $M_n$ with the ultimate result stating that, with $\const_2 = 3/(2s_\rho)$,
	\begin{equation*}
		M_n - \const_1n - \const_2\log n
	\end{equation*} 
	converges in distribution to a random shift of the Gumbel law, see~\cite{13:aidekon:convergence}. One can in fact argue that in this case 
	there is a relatively small portion of the 
	particles from generation $n$ in the neighbourhood of $M_n$ whose children are most likely to be the rightmost one in 
	generation $n+1$ (see Figure~\ref{fig:one}).
	\begin{figure}
		\includegraphics[width=0.5\textwidth]{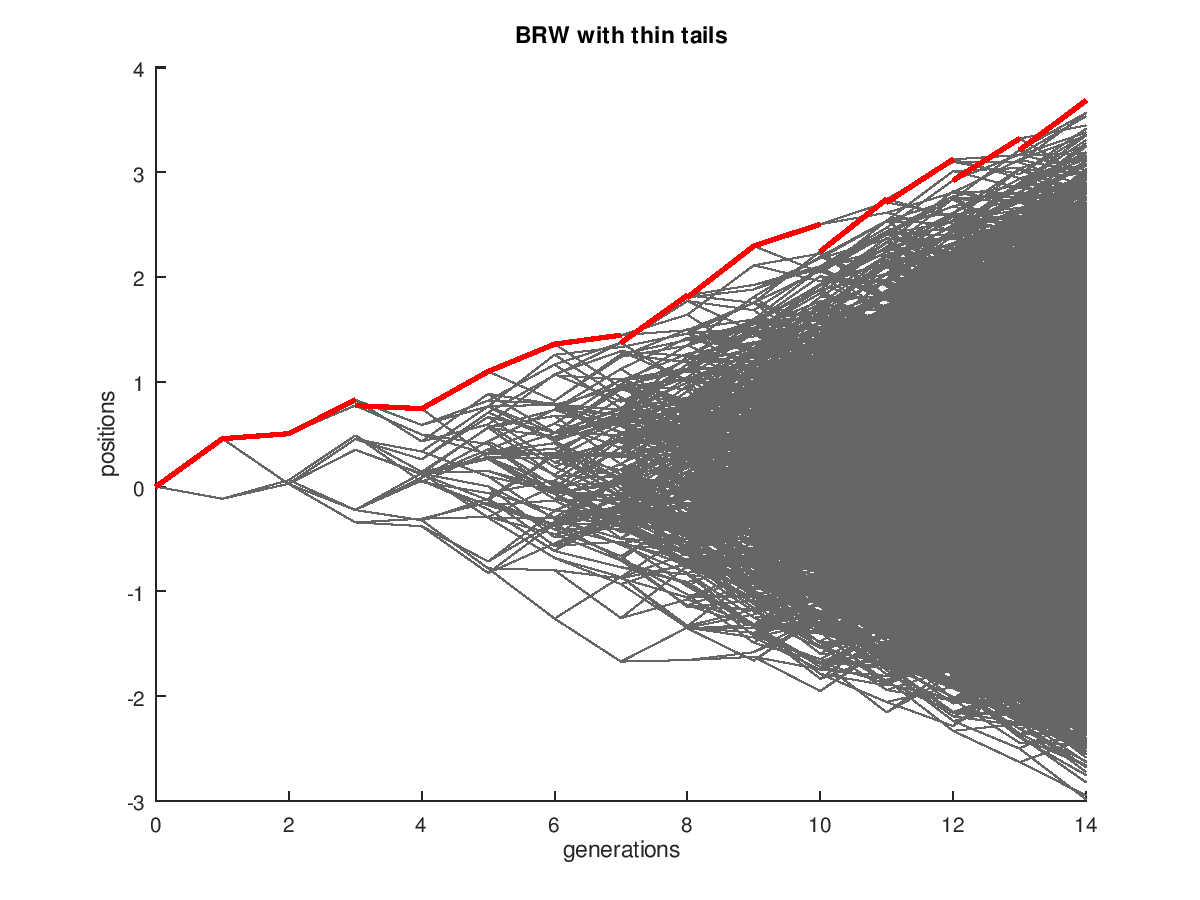}\includegraphics[width=0.5\textwidth]{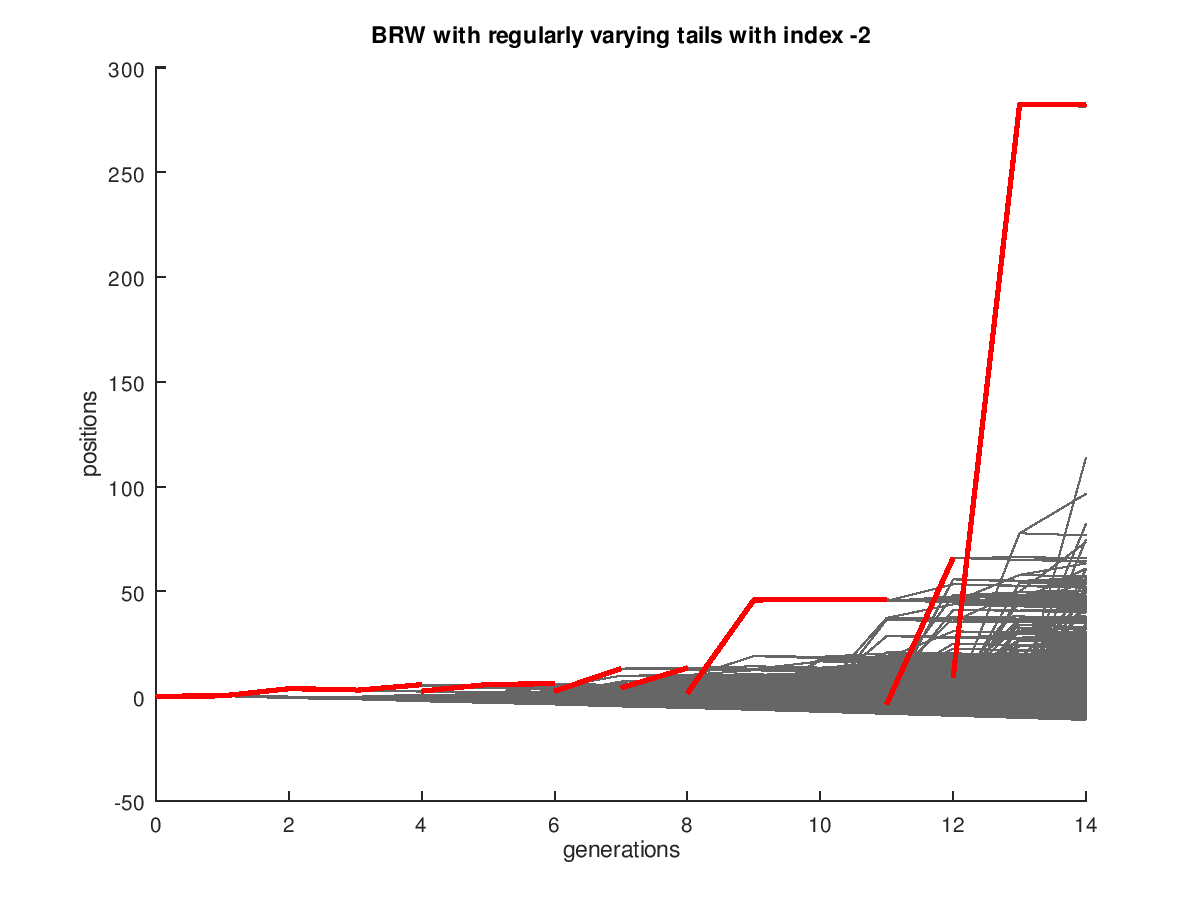}
		\centering
		\caption{BRW with steps distributed uniformly on the interval (-1/2, 1/2) (left) and regularly varying steps with index $\gamma = 2$ (right). 
		The red lines indicate the displacements of the rightmost particles from its place of birth.}
		\label{fig:one}
	\end{figure}
	This is the scale to study the extremes of BRW. Indeed the point process
	\begin{equation}\label{eq:2:pointcramer}
		\shift_{-\const_1n - \const_2\log n}\Lambda_n  
	\end{equation} 
	has a weak limit in $\Mcal_p$, see~\cite{madule2017} for details. It turns out that the limiting measure of~\eqref{eq:2:pointcramer} 
	is a randomly shifted 
	decorated Poisson process~\cite{subag2015freezing}. 
	\begin{definition}\label{def:2:dppp}
		 $\Theta$ is a decorated Poisson point process 
		of intensity $v$ and decoration $D$
		if $\Theta$ is distributed as $\sum_{i \geq 1} \shift_{\zeta_i} D_i$ where $\sum_{i\geq 1}\delta_{\zeta_i}$ is a Poisson point 
		process with 
		intensity $v$ and $D_i$ are i.i.d. copies of a point process $D$. 
	\end{definition}
	
	\begin{definition}\label{def:2:sdppp}
		 $\Theta$ is a randomly shifted decorated Poisson point process 
		of intensity $v$, decoration $D$ and shift $S$ 
		if $\Theta$ is distributed as $\sum_{i \geq 1} \shift_{\zeta_i+S} D_i$ where $\sum_{i\geq 1}\delta_{\zeta_i}$ is a Poisson point 
		process with 
		intensity $v$ and $D_i$ are i.i.d. copies of a point process $D$. One then writes
		$\Theta \sim {\rm SDPPP}(v, D, S)$. 
	\end{definition}
	The weak limit of~\eqref{eq:2:pointcramer} is ${\rm SDPPP}(\const_3 e^{-x} \ud x, N, -\log(D_\infty))$ for some positive $\const_3$, 
	some point process $N$ and $D_\infty$ denoting the a.s. limit of the so-called 
	derivative martingale associated with the BRW~\cite[Chapter 5]{Shi2015}.  \medskip

	The behaviour of $M_n$ under~\eqref{eq:1:regular} is very different. In this case the law of $X$ obeys the "one big jump principle", that is
	\begin{equation}\label{eq:2:onebig}
		\P[S_n> x_n] \sim n \P[X>x_n] 
	\end{equation} 
	provided that $x_n$ is large enough (for example $x_n \gg \sqrt{n\log n}$ if $\gamma>2$). 
	Going back to the estimate~\eqref{eq:2:firstmoment} one can expect, taking into account \eqref{eq:2:onebig} 
	and considering $x_n = m^{n/\gamma}$, that
	\begin{equation*}
		m^{-n/\gamma} M_n 
	\end{equation*}
	converges in law. Indeed, this sequence converges in law to a random shift of the Fréchet distribution, see~\cite{Durrett1983}. 
	In this case the rightmost particles are most likely to originate from the bulk of the population, i.e. the exponential number of 
	typically placed particles (see Figure~\ref{fig:one}).
	Similarly as before, one has convergence of the associated point process and one can show that 
	after an appropriate scaling $\Lambda_n$ converges to so-called Cox cluster process, see~\cite{bhattacharya:2017:point}. 
	Due to the similarity with our main result we now describe the limiting point process in detail.
	Assume for simplicity that the displacements are positive and let $\{T_l\}_{l \in \N}$ be a collection of i.i.d. random variables with 
	common distribution given by
	\begin{equation}\label{eq:2:tk}
		\P[T_1 = k ] = \frac{1}{\upsilon} \sum_{i=0}^\infty m^{-i} \P[Z_i=k], \quad k=1,2, \ldots ,  
	\end{equation} 
	where 
\begin{equation}\label{vdef}
\upsilon =  \sum_{i=0}^\infty m^{-i} \P[Z_i>0]\, .
\end{equation}
	Next let $\{g_l\}_{l\in \N}$ denote the points of a Poisson point process with intensity\\
 $\gamma x^{-\gamma-1}\1_{(0,\infty)}(x)\ud x$. Then  
	\begin{equation*}
		\scale_{m^{-n/\gamma}} \Lambda_n \Rightarrow \sum_{l=0}^\infty T_l \delta_{(\upsilon W)^{1/\gamma} g_l},
	\end{equation*}
	in $\Mcal_p$, where $W$ is given by \eqref{Wdef}. The coefficients $\{ T_l\}_{l \in \N}$ represent the number of 
	descendants of particles that made a big jump.
	As we will see, the limiting point process in the case of stretched exponential displacements has a very similar representation. 
	
\section{Main result}\label{sec:main}	 
	 
	We turn our attention to a case between~\eqref{eq:1:cramer} and~\eqref{eq:1:regular}, namely the case where the displacements have a 
	stretched exponential tail. 
	We will impose some regularity condition on the regularly varying function $R(x) = x^r\ell(x)$ that appears in the tail of the displacement law. 
	One can usually circumvent this kind of 
	restrictions using the smooth variation theorem~\cite[Theorem 1.8.2]{bingham_goldie_teugels_1987}. However this is not sufficient for our 
	proposes since we would 
	need an approximation up to a certain order. We will thus impose sufficient smoothness on $R$ from the beginning.
	From now on we will assume that the law of the displacements is given as follows.  

	\begin{as} \label{as:RW}
		The random variable $X$ is centred $(\E [ X ] =0)$, has variance $1$ $(\E\left[X^2 \right]=1)$ and has a stretched exponential 
		upper tail with index 
		$r \in (0,1)$, that is for some positive $x_0$,
		\begin{equation*}
                  \P[X > x]=a(x)e^{-R(x)}, \qquad x>x_0,
		\end{equation*}
		where $a, R \colon (x_0, +\infty) \to [0, \infty)$, $a(x) \to a>0$ as $x \to \infty$ and $R$ is a $\mathcal{C}^1(x_0, +\infty)$ function 
		such that $R'$ is regularly varying with index $r-1$ and differentiable with $R''(x) = O(x^{r-2+\epsilon})$ for any $\epsilon>0$. 
		Moreover we assume that 
		$R(x) x^{-d}$ is eventually 
		decreasing for some $d\in (0,1)$ and if $r>2/3$ we assume additionally that $R(x) x^{-1/2-\epsilon}$ is eventually 
		increasing and $R(x)x^{-1+\epsilon}$ 
		is eventually decreasing for some $\epsilon>0$.  Further, suppose that 
		\begin{equation*}
			 \E[|X|^k]<\infty
		\end{equation*}
		for all $k \leq \left\lfloor (2-r)/(1-r)\right\rfloor$.
	\end{as}

	We note right away that since $a(\cdot)$ is assumed to be bounded away from zero, the condition regarding regularity of $R$ 
	excludes discrete random variables,
	i.e. taking values in a closed, countable set. 
	Recall that $R'$ is regularly varying at infinity with index $r-1$, if it can be written in the form 
	\begin{equation*}
		R'(x) = x^{r-1}\ell_{R'}(x),\quad x> x_0
	\end{equation*} 
	for some slowly varying function $\ell_{R'}$. 
	By Karamata's Theorem~\cite[Theorem 0.6]{resnick:2013:extreme} this implies that $R$ is also regularly varying 
	with index $r$, i.e. it is of the form 
	\begin{equation*}
		R(x)= x^r \ell_{R}(x), \qquad x>x_0
	\end{equation*}
	for some slowly varying function $\ell_R$. There is a natural relation between $\ell_R$ and $\ell_{R'}$, but it will be of 
	little importance for us and 
	thus we refrain from stating it explicitly. 
	Note that under Assumption~\ref{as:RW}, $R''(x)/R'(x)^2 \to 0$ as $x \to \infty$ and thus $e^{-R(x)}$ is a Von Mises function with an auxiliary 
	function $1/R'(x)$ \cite[Proposition 1.1]{resnick:2013:extreme}. 
	This further implies that under Assumption~\ref{as:RW} 
	the distribution of $X$ lies in the max-domain of attraction of the Gumbel law~\cite[Proposition 1.4]{resnick:2013:extreme}. 
	We will use this fact on the exponential scale. For sufficiently large $n$ we put
	\begin{equation}\label{eq:3:dnDef}
		d_n = \inf\{ x \in (x_0, \infty) \: : \: R(x)> n \log m \}.
	\end{equation}
	Note that by continuity of $R$, for sufficiently large $n$, $d_n$ is the smallest solution to $R(d_n) = n \log m $.
	The sequence $\{d_n\}_{n \in \N}$ is regularly varying with index $1/r$, i.e.
	\begin{equation*}
		d_n = n^{1/r} \ell_d(n)
	\end{equation*}
	for some slowly varying function $\ell_d$~\cite[Proposition 0.8]{resnick:2013:extreme}. Let also 
	\begin{equation}\label{eq:3:anDef}
		a_n = 1/R'(d_n)
	\end{equation}
	for $n$ large enough.
	Since a composition of two regularly varying function is regularly varying, $a_n$ admits the representation 
	\begin{equation}\label{whatisa_n}
		a_n = n^{1/r-1}\ell_a(n)
	\end{equation}
	for some slowly varying function $\ell_a$. Extreme value theory asserts that, under Assumption~\ref{as:RW}, the maximum of 
	$m^n$ independent copies of $X$ shifted by $d_n$ and 
	scaled by $a_n$ converges in distribution to the Gumbel 
	law, that is for any $x \in \R$, as $n \to \infty$,
	\begin{equation*}
		\P[X_k \leq d_n +a_nx, \: k\leq m^n] \to e^{-ae^{-x}}.
	\end{equation*}
	Note that here one uses the fact that $R''(x)/(R'(x))^2 \to 0$ as $x \to \infty$. As we will soon see, the sequences 
	$\{d_n\}_{n \in \N}$ and $\{a_n\}_{n \in \N}$ will drive the centring and 
	scaling for $\Lambda_n$. In fact the latter sequence is the scaling for $\Lambda_n$.\medskip

	Large deviations for random walks under Assumption~\ref{as:RW} date back to the late sixties, 
	see~\cite{1969:nagaev:integral, 1969:nagaev:integral2, rozovskii1994probabilities, linnik1961limit} or~\cite{Denisov2008} for a 
	general set-up. 
	As it turns out~\eqref{eq:2:onebig} is satisfied also in this case provided that 
	$x_n \gg n^{1 /(2-2r)+\varepsilon}$ for some $\varepsilon>0$. Such deviations are said to lie in the one big jump domain. 
	Once again the estimate~\eqref{eq:2:firstmoment} 
	suggests that in this case $M_n$ should be of order 
	$d_n = n^{1/r}\ell_d(n)$. 
	This causes some technical problems since $d_n \gg  n^{ 1/(2-2r)+\varepsilon}$ if and only if $ r < 2/3$. 
	Indeed for $r \geq  2/3$ and $x_n$ of the order $d_n$, the asymptotic of $\P[S_n>x_n]$ differs from the one of 
	$n\P[X>x_n]$. However the latter still provides the leading term in the asymptotic of the former. More precisely, for any $r \in (0,1)$, 
	\begin{equation}\label{logas}
		\log \P\left[S_n>d_n\right] = -(\log m) n (1+o(1)),
	\end{equation}
	see~\cite[Theorem 3]{Gantert2000}. In the sequel we will revisit this problem to provide a new description of the classical 
	results~\cite{1969:nagaev:integral} on the precise asymptotic of $\P[S_n > x_n]$ with 
	$x_n \sim  d_n$. The~logarithmic asymptotics \eqref{logas} suffices to predict the law of large numbers for $M_n$. 
	Indeed, as shown in~\cite{Gantert2000} we have
	\begin{equation*}
		\lim_{n \to \infty }M_n/ d_n = 1 \qquad \P^*- \text{ a.s. } 
	\end{equation*}	
	The behaviour of $M_n - d_n$ was recently studied in \cite{dyszewski:max:2022} for the special case $R(x) = \lambda x^r$ for some $\lambda>0$. 
	In this case 
	\begin{equation*}
		d_n = ( ( \log m) /\lambda)^{1/r} n^{1/r} \quad \mbox{and} \quad a_n = r ((\log m)/\lambda)^{1/r-1}n^{1/r-1}. 
	\end{equation*}	
	It transpires that for $r \leq 2/3$ the rightmost particle is most likely to originate from the bulk of the population (see Figure~\ref{fig:2}) 
	and for $r>2/3$ it typically originates from a subexponential number of particles that deviate from the bulk (located around $O(n^{2-1/r})$), 
	see Figure~\ref{fig:2}. This in turn leads to a more balanced behaviour of $M_n$ for big values of $r \in (0,1)$.
	\begin{figure}
		\centering
		\includegraphics[width=0.5\textwidth]{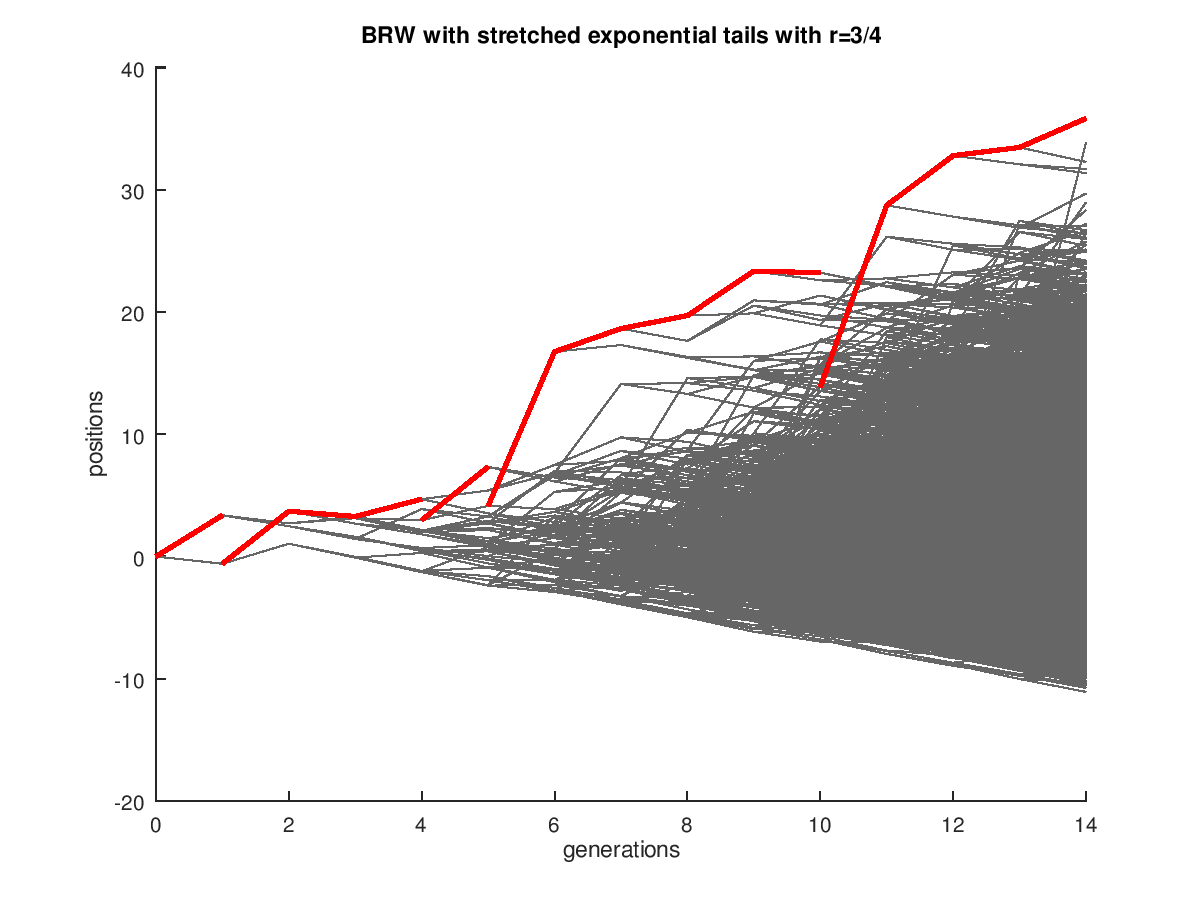}\includegraphics[width=0.5\textwidth]{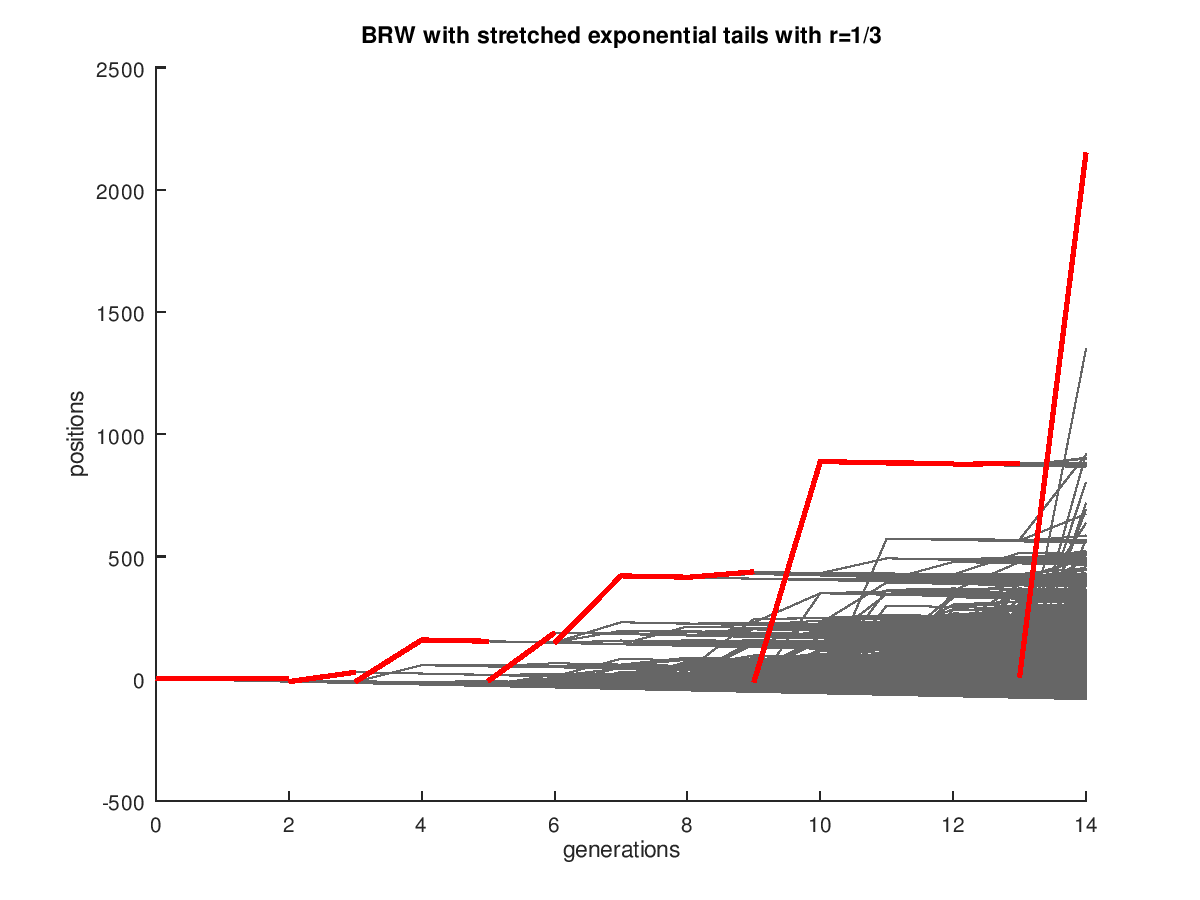}
		\caption{BRW with stretched exponential steps with $r=3/4$ (left) and $r=1/3$ (right). The red lines indicate the displacements of the rightmost particles from its place of birth.}
		\label{fig:2}
	\end{figure}
	More precisely, for this choice of $R$ we have, for $r > 2/3$,
	\begin{equation*}
		\lim_{n \to \infty }\frac{M_n -  ((\log m)/\lambda)^{1/r} n^{1/r} }{n^{ 2 - 1/r} } = \frac{r\lambda^{1/r}(\log m)^{1-1/r}}{2 }  \quad \P^*- \text{ a.s.}
	\end{equation*} 
	and for $r \leq 2/3$, with respect to $\P^*$,
	\begin{equation*}
		\frac{M_n - ((\log m)/\lambda)^{1/r} n^{1/r}}{ n^{ 1/r-1}}
	\end{equation*} 
	converges in distribution to a random shift of the Gumbel law~\cite[Theorem 3.1]{dyszewski:max:2022}.
	Here, we extend these results to distributions satisfying Assumption~\ref{as:RW}
	and prove point process convergence. 
	The limit object of $\Lambda_n$ will be described with the help of an auxiliary polynomial. However it comes into play only in the case $r\geq 2/3$. 
	Let the sequence of cumulants of $X$, $\{k_j\}_{j\leq \kappa}$ be defined recursively via $k_1=0$ and for $j\geq 1$, 
	\begin{equation}\label{eq:4:c1}
			k_{j} = \E\left[X^j\right] - \sum_{i=1}^{j-1} \binom{j-1}{ i} k_{j-i}\E\left[X^i\right]. 
	\end{equation}
	Define the truncated cumulant generating function via
	\begin{equation}\label{eq:2:truncatedcumulant}
		K(x) = \sum_{j=2}^\kappa \frac{k_j}{j!}x^j, \qquad \kappa = \left\lfloor \frac{2-r}{1-r}\right\rfloor.
	\end{equation}
	The name truncated cumulant generating function will become clear after we state Lemma~\ref{lem:4:asK} below.
	For $n \in \N$ consider a function $\Psi_n$ given for $z \in \R$ and sufficiently large $n$ via
	\begin{equation*}
		\Psi_n ( z ) = \inf_{s \in[0,1]}  \{ R(d_n+z-nK'(s)) +n(sK'(s)-K(s))  \}.
	\end{equation*}
	One can then check that $\Psi_n$ is a nondecreasing and continuous function such that for any $n \in \N$, $\Psi_n(z) \to \infty$
	as $z \to \infty$.
	Finally, since $\Psi_n(0) \leq R(d_n)$, for sufficiently large $n \in \N$ we can define $\tau_n$ to be the smallest real number solving
	\begin{equation}\label{eq:3:tau}
		\Psi(\tau_n) = R(d_n).
	\end{equation}
	By a direct calculation using the regular variation of $R$ and $R'$, one gets that for some slowly varying function $\ell$,  
	\begin{equation*}
		\tau_n = nR'(d_n)/2 + O(n^{3-2/r}\ell(n)).
	\end{equation*}
	It follows that $\tau_n$ is regularly varying with index $2-1/r$ and thus, recalling \eqref{whatisa_n}, $\tau_n = o(a_n)$ if $r< 2/3$.
	We will now describe the limiting measure. Let $\{T_k\}_{k \in \N}$ be i.i.d. with distribution given in~\eqref{eq:2:tk} and take 
	$\{\iota_k\}_{k\in \N}$ to be the points of a Poisson point process with intensity $a e^{-x} \ud x$. Define 
	a random element of $\Mcal_p$ by
	\begin{equation}\label{eq:3:LambdaLimit}
		\Lambda = \1_{\{ W>0 \}}\sum_{k=0}^\infty T_k \delta_{ \iota_k -\log(\upsilon W)},
	\end{equation}
	recalling \eqref{vdef}.
	Then $\Lambda$ is a Cox cluster process with the Laplace functional given via
	\begin{equation*}
			\E \left[ e^{ - \int f(x) \: \Lambda(\ud x)} \right] =  
			\E\left[\exp \left\{ - aW \sum_{j=0}^\infty m^{-j}  \int_{\R} \E \left[ 1- e^{-Z_jf(x)}  \right]  e^{-x}\ud x  \right\}\right]
	\end{equation*} 
	for $f \in \Ccal^+_K$. Note that even though $W$ is the limiting random variable of $Z_j/m^j$, the $Z_j$'s present in the above formula
	can be taken independent from $W$ due to the presence of the expectation in the exponent. If we denote by 
	$z \colon [0,1] \to [0,1]$ the probability generating function of the underlying reproduction law 
	\begin{equation*}
		z(s) = \E\left[ s^{Z_1} \right]
	\end{equation*} 
	then a standard application of the branching property yields that the probability generating function of $Z_n$ is an $n$-fold composition
	\begin{equation*}
		\E\left[ s^{Z_n} \right] = z^{(\circ n)}(s) = z\circ z \circ \ldots \circ z(s), \qquad s\in [0,1].
	\end{equation*}
	The formula for the Laplace functional of $\Lambda$ can be thus recast as 
	\begin{equation*}
		\E \left[ e^{ - \int f(x) \: \Lambda(\ud x)} \right] =  
		\E\left[\exp \left\{ - aW \sum_{j=0}^\infty m^{-j}  \int_{\R} \left( 1-  z^{(\circ j)}\left( e^{-f(x)} \right)  \right)  e^{-x}\ud x  \right\}\right].
	\end{equation*} 

	\begin{theorem}\label{thm:2:main}
		Suppose that Assumptions~\ref{as:BP} and~\ref{as:RW} are in force. Let $d_n$ and $a_n$ be given by \eqref{eq:3:dnDef} 
		and \eqref{eq:3:anDef} respectively and take $b_n = d_n +\tau_n$, 
		where $\tau_n$ is given in~\eqref{eq:3:tau}.
		Then 
		\begin{equation*}
			\scale_{a_n^{-1}}\shift_{-b_n}\Lambda_n  \Rightarrow \Lambda \quad \mbox{in } \Mcal_p,
		\end{equation*}
 		with $\Lambda$ given by~\eqref{eq:3:LambdaLimit}.
	\end{theorem}

	Note that for any $x \in \R$, $\P[M_n \leq x] = \P[\Lambda_n(x, +\infty]=\emptyset]$.
	As a by-product of Theorem~\ref{thm:2:main} we obtain a new result for $M_n$ which complements the results 
	of~\cite[Theorem 3.1]{dyszewski:max:2022}. 

	\begin{corollary}\label{cor:main}
		Let Assumptions~\ref{as:BP} and~\ref{as:RW} hold. Then for $M_n = \max_{|w|=n} V(w)$ and $b_n$ and $a_n$ as in the statement of 
		Theorem~\ref{thm:2:main},
		\begin{equation*}
			\frac{M_n - b_n  }{ a_n} \overset{d}\to V
		\end{equation*}
		with respect to $\P$, where the random variable $V \in [-\infty, \infty)$ is distributed according to
		\begin{equation*}
			\P[V \leq x] = \E \left[ \exp \left\{ -  a\upsilon W e^{-x} \right\} \right] .
		\end{equation*}
	\end{corollary}
Note that $V$ can have an atom at $-\infty$ if $W$ has an atom at $0$.
	The point process convergence allows to treat upper order statistics of the positions of particles in generation $n$. For $k \in \N$, let
	$M_n^{(k)}$ denote the $k$-th order statistic of $\{V(w)\}_{w \in \mathbb{T}_n}$. That is for any $n \in \N$
	$\{M_n^{(k)}\}_{k \leq \# \mathbb{T}_n}$ is the non-decreasing enumeration of $\{V(w)\}_{w \in \mathbb{T}_n}$. Then we have
	\begin{equation*}
		\P\left[ M_n^{(k)} \leq b_n+a_nx \right] = \P[ \scale_{1/a_n}\shift_{-b_n}\Lambda_n(x, +\infty] \leq k-1 ] \to 
		\P[\Lambda(x, +\infty]\leq k-1].
	\end{equation*}
	Following~\cite{bhattacharya:2017:point} 
	we can express the limiting quantity on the right hand side in the following way. Consider a marked point process
	\begin{equation*}
		\tilde\Lambda = \sum_{k=0}^\infty \delta_{T_k} \otimes\delta_{ \iota_k -\log(\upsilon W)},
	\end{equation*}
	which conditioned on $W$ is a Poisson point process with intensity 
	\begin{equation*}
		a\upsilon W\P[T_1 \in \ud x] \otimes e^{-y}\ud y.
	\end{equation*}	
	Denote by $\pi$ a generic partition of some given integer, say $l$, of the form $l=i_1y_1+\ldots + i_{|\pi|}y_{|\pi|}$, where
	each $i_j = i_j(\pi)$ repeats $y_j=y_j(\pi)$ times in the partition and the $i_j$ are increasing.
	Here $|\pi|$ denotes the length of the partition, i.e. the number of its distinct elements.
	Write $\Pi_l$ for the collection of all possible partitions of $l \in \N$.
	Then $\P[\Lambda(x, +\infty]\leq k-1]$ is equal to
	\begin{multline*}
		\sum_{l=1}^{k-1} \sum_{\pi \in \Pi_l} \P \left[ \bigcap_{j=1}^{|\pi|} \{ \tilde \Lambda (\{ i_j\}\times (x +\infty] = y_j \} \right]
		=\\ \sum_{l=1}^{k-1} \sum_{\pi \in \Pi_l } 
		\E \left[ \prod_{j=1}^{|\pi|} (a \upsilon W e^{-x} \P[T_1=i_j])^{y_j} \frac{1}{y_j!} e^{-a \upsilon W  e^{-x}\P[T_1=i_j]} \right].
	\end{multline*}
	The right hand side is therefore the limit of $\P[M_n^{(k)} \leq b_n+a_nx]$. 

	\begin{remark}
		The limiting point process~\eqref{eq:3:LambdaLimit} falls into the class of randomly shifted decorated Poisson point processes. 
		Comparing the Definition~\ref{def:2:sdppp} with~\eqref{eq:3:LambdaLimit} one readily sees that
		\begin{equation*}
			\Lambda \sim {\rm SDPPP} \left( a e^{-x} \ud x, T_1 \delta_0,- \log(\upsilon W) \right).
		\end{equation*}
	\end{remark}

	\begin{remark}
		The limiting point process $\Lambda$ enjoys the superposability~\cite[subsection 3.4]{brunet2011branching}: the distances from the 
		rightmost position of a union of realizations of $\Lambda$ have the same distribution as those of a 
		single realization. 
		To state this in a more rigorous way let 
		$\Lambda^{(i)} = \sum_{k=0}^\infty T_k^{(i)} \delta_{ \ell_k^{(i)} -\log(\upsilon W_i)}$
		denote i.i.d. copies of $\Lambda$. Since the martingale limit $W$ is known to satisfy
		\begin{equation*}
			\sum_{i=1}^{Z_1} W_i \stackrel{d}{=} mW
		\end{equation*}
		where $W_1, W_1, \ldots $ are i.i.d. copies of $W$ independent of $Z_1$, one sees from studying the corresponding Laplace functionals that
		\begin{equation*}
			\sum_{i=1}^{Z_1} \Lambda_i \stackrel{d}{=} \shift_{ \log m}\Lambda.
		\end{equation*}
		Now consider a mapping $\tip$ on $\Mcal_p$ given by 
		\begin{equation*}
			\tip \sum_{i}\delta_{x_j} = \left\{ \begin{array}{lr}  \sum_{i}\delta_{x_i - \max_jx_j} & \max_jx_j <\infty \\ 
			o & \max_jx_j=\infty, \end{array} \right.
		\end{equation*}
		where $o$ denotes the null measure.
		Then $\tip \eta$ is the point process $\eta$ seen from the rightmost particle. 
		Since $\tip \circ \shift_x = \tip$ for any $x \in \R$, one readily sees that
		\begin{equation*}
			\tip\sum_{i=1}^{Z_1} \Lambda_i \stackrel{d}{=} \tip \Lambda.
		\end{equation*}
	\end{remark}
	\begin{remark}
		One can show in a similar fashion that $\Lambda$ is exponentially stable provided that $W=1$ a.s. which is the case when the branching is deterministic, i.e.
		$\P[Z_1=k]=1$ for some $k\geq 2$. More precisely $\Lambda$ is exponentially $1$-stable if for any choice of $\alpha, \beta \in \R$
		such that $e^\alpha+e^\beta=1$ we have
		\begin{equation*}
			\shift_\alpha \Lambda + \shift_\beta \Lambda' \stackrel{d}{=} \Lambda, 
		\end{equation*}
		where $\Lambda'$ is an independent copy of $\Lambda$. One can check that this is true since in the case $W=1$ a.s. one has
		\begin{equation*}
			\E \left[ e^{- \int f(x) \shift_\alpha\Lambda(\ud x)} \right] 
			= \exp\left\{-a (k-1)e^{\alpha} \int_\R\sum_{j=0}^\infty k^{-j-1}\left(1-e^{-k^jf(x)}\right)  e^{-x}\ud x \right\}
		\end{equation*}
		for any $\alpha\in \R$. It has been conjectured in~\cite[subsection 6.1]{brunet2011branching} and proved in~\cite{maillard2013note} that a 
		exponentially $1$-stable point process is a decorated exponential Poisson process i.e. with intensity of the form $\const e^{-x}\ud x$ and a certain
		decoration. This can be checked quite easily in our context since the aforementioned $SDPPP$ representation of $\Lambda$ boils down to a $DPPP$
		representation
		\begin{equation*}
			\Lambda = \sum_{k} \shift_{\iota_k} D_k, \qquad D_k = T_k\delta_{-\log \upsilon}.
		\end{equation*}
with $\upsilon$ defined in \eqref{vdef}.
	\end{remark}

\section{Trimming lemmas}\label{sec:trimming}
	In our first steps towards the proof of Theorem~\ref{thm:2:main} we will show which particles in $\mathbb{T}_n$ contribute to the extremes of $\Lambda_n$. 
	To facilitate this we will
	use several deviation estimates for sums of i.i.d. random variables with stretched exponential tails. We will often use the following expansion for $R$,
	\begin{equation*}
		R\left( d_n + x\right) =R(d_n) + R'(d_n)x+x^2R''(\xi_{d_n,d_n+x})/2 
	\end{equation*}
	for $x \in \R$ and some $\xi_{d_n,d_n+x}$ that lies between $d_n$ and $d_n+x$.	

	The first few steps of our arguments concentrate on pinpointing the particles that contribute to the point process which is the limit of 
	$\scale_{a_n^{-1}}\shift_{-b_n}\Lambda_n $. 
	As we will see the positions of those particles will be asymptotically independent. 
       We will partition $\mathbb{T}_n$ into four classes of particles and show which class contributes to the asymptotics of $\scale_{a_n^{-1}}\shift_{-b_n}\Lambda_n $.  
       The first one consists of those particles with no big jumps along their ancestral line, i.e.
	\begin{equation*}
		\mathcal{A}_n = \left\{w \in \mathbb{T}_n\: :\: \max_{ z \in (\emptyset, w]}  X_z \leq \varepsilon_n \right\},  
	\end{equation*} 
	where $\varepsilon_n =\varepsilon d_n$ for some fixed $\varepsilon \in \left(2^{-r},1\right)$. 
	The positions of particles from $\mathcal{A}_n$ consist of individual displacements of 
	typical particles and are therefore well controlled. 
	The point process given by the positions of the particles from $\mathcal{A}_n$ will be denoted by
	\begin{equation*}
		\Lambda_n^{\mathcal{A}}=\sum_{w \in \mathcal{A}_n} \delta_{ a_n^{-1}(V(w) - b_n)}.
	\end{equation*}
	As we will see in Lemma~\ref{lem:3:an} below, the contribution of particles from $\mathcal{A}_n$ is negligible. 
	To control the contribution of $\Lambda_n^{\mathcal{A}}$ we will need a deviation results for typical particles. For $L>0$ take 
	$\P_{L}$ to be a probability measure such that the $X_k$'s are i.i.d. under $\P_{L}$ with distribution given by
	\begin{align*}
		\P_{L}\left[ X_k \leq t \right] & = \P\left[\:X_k\leq t\: \left| \: X_k\leq L \: \right. \right].
	\end{align*}	
	One can use classical large deviation techniques to control the sum of the $X_k$'s under $\P_L$. 
	More precisely one can exploit the exponential change of measure and use the 
	cumulant generating function of the truncated variables
	\begin{equation*}
		K_L(s) =  \log \E_{L} \left[  \exp\{ sX_1\} \right], \quad s >0,
	\end{equation*}
	where $\E_{L}$ denotes the expectation corresponding to $\P_{L}$. In order to be able to control the particles from $\mathcal{A}_n$ we need to take 
	$L=\varepsilon_n$ in the above construction. 
	We will first check the asymptotic expansion of 
	$K_L=K_{\varepsilon_n}$ near $0$. For functions $f$, we write $f^{(0)}$, $f^{(1)}$ and $f^{(2)}$ for $f$, $f'$ and $f''$ respectively. 
	Recall the truncated cumulant generating function $K$ given via~\eqref{eq:2:truncatedcumulant}.

	\begin{lemma}\label{lem:4:asK} 
		Let Assumption~\ref{as:RW} be in force and let $\varepsilon_n = \varepsilon d_n$ for some $\varepsilon\in (2^{-r},1)$. 
		Let $d\in (0,1)$ be such that $R(x)x^{-d}$ is decreasing. Then for $\chi \in (0, 1/d)$ and $i \in \{0, 1,2\}$, 
		\begin{equation*}
			\sup_{0 \leq s \leq  \chi R'(d_n) } \left| K_{\varepsilon_n}^{(i)}(s) - K^{(i)}(s) \right| = o \left( n^{-1}\right).
		\end{equation*}
	\end{lemma}

	The proof of Lemma~\ref{lem:4:asK} is quite standard and utilizes the fact that the $k_m$'s obey a~similar recursive formula~\eqref{eq:4:c1} as the cumulants 
	of $X_1$ under $\P_{\varepsilon_n}$. 
	The details are given in the appendix.  
	One of the consequences of Lemma~\ref{lem:4:asK} is the fact that since $K(s) \sim s^2/2$ as $s \to 0^+$ we have
	\begin{equation*}
		nK_{\varepsilon_n}(s) = (1+o(1))ns^2/2,
	\end{equation*}
	as $n\to \infty$ and $s \to 0^+$ provided that $s \leq \chi R'(d_n)$.
	For future reference denote
	\begin{equation}\label{eq:4:defN}
		N_n = \max_{1 \leq k \leq n} X_k.
	\end{equation}
	With Lemma~\ref{lem:4:asK} at hand we can state and prove the aforementioned lemma regarding the contribution of particles from $\mathcal{A}_n$. 

	\begin{lemma}\label{lem:3:an}
		Under Assumptions~\ref{as:BP} and \ref{as:RW} for any $r \in (0,1)$, if $b_n \sim d_n$ and $a_n =  1/R'(d_n)$, 
		\begin{equation*}
			\Lambda_n^{\mathcal{A}} \Rightarrow o
		\end{equation*}
		in $\Mcal_p$, where $o$ denotes the null measure.
	\end{lemma}
	
	\begin{proof}
		Since the convergence is tested by integrating functions of compact support, it is enough to show that
		\begin{equation*}
			\Lambda_n^{\mathcal{A}} [L, \infty]  \overset{\P}\to 0
		\end{equation*}
		for any constant $L\in \R$. Since $\Lambda_n^{\mathcal{A}}$ is integer valued it is sufficient to estimate
		\begin{align*}
			\P \left[ \Lambda_n^{\mathcal{A}} [L , \infty] \geq 1 \right] \leq \E \left[ \Lambda_n^{\mathcal{A}} [L , \infty] \right] = 
			\E[Z_n] \P\left[S_n > b_n + L a_n, \: N_n \leq \varepsilon_n\right].
		\end{align*} 
		To tackle the second term on the right hand side note that
		\begin{equation*}
			\P\left[S_n > b_n + L a_n, \: N_n \leq \varepsilon_n\right] \leq \P_{\varepsilon_n}\left[S_n > b_n + L a_n\right].
		\end{equation*}
		Now use the exponential Markov inequality to get for $s>0$,
		\begin{equation*}
			\P_{\varepsilon_n}\left[S_n > b_n +L a_n\right] \leq \exp \left\{ -s (b_n + L a_n) + nK_{\varepsilon_n}(s) \right\}.
		\end{equation*}
		Choose $s = (n\log m + n^\xi)/(b_n + L  a_n)$ with $\xi \in (3 - 2/r,1) \cap (0,1)$. By Lemma~\ref{lem:4:asK}, $nK_{\varepsilon_n}(s) = o(n^{\xi})$ and thus
		\begin{equation}\label{eq:4:dtrongerclaim}
			\exp \left\{ -s (b_n + L a_n) + nK_{\varepsilon_n}(s) \right\} \leq m^{-n}\exp \left\{ -  n^\xi/\const  \right\}
		\end{equation}
		which yields our claim.
	\end{proof}
	
	The next class of particles consist of those that had at least two big jumps along the ancestral line, that is
	\begin{equation*}
		\mathcal{B}_n = \left\{w \in \mathbb{T}_n\: : \:   \exists v, u \in (\emptyset, w], \: v \neq u, \:  \: \min \{X_v, X_u\} > \varepsilon_n \right\}.
	\end{equation*} 		
	By the choice of $\varepsilon\in (2^{-r},1)$, $\mathcal{B}_n =\emptyset$ almost surely for sufficiently large $n$. Indeed, by a direct first moment bound
	\begin{equation}\label{eq:4:bn}
		\P[|\mathcal{B}_n|\geq 1] \leq \const \:  n^2 \: m^n \P\left[X> \varepsilon_n\right]^2 =o(1),
	\end{equation}
	since $2R(\varepsilon_n) \sim 2\varepsilon^r R(d_n)$ and $2 \varepsilon^r>1$.
	Thus
	\begin{equation*}
		\Lambda_n^{\mathcal{B}}=\sum_{w \in \mathcal{B}_n} \delta_{ a_n^{-1}(V(w) - b_n)} \Rightarrow o
	\end{equation*}
	in $\Mcal_p$.
	Particles that are not in $\mathcal{A}_n \cup \mathcal{B}_n$  have exactly one big jump along their ancestral line. 
	We will need to distinguish whether this jump is greater or smaller than
	\begin{equation}\label{eq:4:yn}
		y_n=\left\{ \begin{array}{rr}
				d_n - (\log n)^2 \frac{d_n^2+nR(d_n)^2}{ d_nR(d_n)} & r\in \left(0,  \frac 23 \right] \\
				d_n -  n R'(d_n)& r \in \left( \frac 23, 1 \right) .		
			\end{array}\right.
	\end{equation}
	For $w \in \mathbb{T}$ define $v^*(w)$ to be the element of $(\emptyset, w]$ closest to the root such that $X_{v^*(w)} = \max_{z \in (\emptyset,w]}X_z$. 
	The next class of particles we will consider is given by
	\begin{equation*}
		\mathcal{C}_n = \left\{w \in \mathbb{T}_n \: : \:  X_{v^*(w)} \in \left(\varepsilon_n, y_n\right] \mbox{ and }  
		\max _{ u \in (\emptyset, w]\setminus \{v^*(w)\}}  \:  X_{u} \leq \varepsilon_n \right\}.
	\end{equation*}		
	Define the corresponding point process via
	\begin{equation*}
		\Lambda_n^{\mathcal{C}}=\sum_{w \in \mathcal{C}_n} \delta_{ a_n^{-1}(V(w) - b_n)}.
	\end{equation*}

	The second truncation sequence $\{y_n\}_{n \in \N}$ is chosen such that the contribution of the particles in $\mathcal{C}_n$ is also negligible. 
	\begin{lemma}\label{lem:3:cr<1}
		Suppose that Assumptions~\ref{as:BP} and \ref{as:RW} are satisfied.
		Then for sequences $\{a_n\}_{n \in \N}$ and $\{b_n\}_{n \in \N}$ chosen as in Theorem~\ref{thm:2:main} we have 
		\begin{equation*}
			\Lambda_n^{\mathcal{C}} \Rightarrow o
		\end{equation*}
		in $\Mcal_p$.
	\end{lemma}

	\begin{proof}
		We will proceed similarly as in the proof of Lemma~\ref{lem:3:an}. In order to carry out this program we will first estimate
		\begin{equation*}
			P_n=\P\left[S_{n-1}+X_n > b_n + L a_n, \: N_{n-1} \leq \varepsilon_n,  \:  X_n \in \left(\varepsilon_n, y_n\right]\right],
		\end{equation*}
		for $L\in \R$ and $N_n$ given in~\eqref{eq:4:defN}. We will first treat the case $r\leq  2/3$. 
		Note that with $h_n = R(d_n)/d_n$, using the exponential Markov inequality, 
		\begin{equation}\label{eq:4:Pnbound}
			P_n  \leq \const e^{-h_n(b_n+La_n) +nK_{\varepsilon_n}\left(h_n\right) } 
			\E \left[ e^{h_nX} \1_{\left\{ X \in \left(\varepsilon_n, y_n\right] \right\}} \right]. 
		\end{equation}
		By Lemma~\ref{lem:4:asK}, $nK_{\varepsilon_n}(h_n) =O(nR(d_n)^2/d_n^2)$ and the expectation on the r.h.s. of \eqref{eq:4:Pnbound}, by 
		Lemma~\ref{lem:a:Hbound} from the appendix, has an upper 
		bound of the form 
		\begin{equation*}
			 \E \left[ e^{h_n X} \1_{\left\{ X \in \left(\varepsilon_n, y_n\right] \right\}} \right] \leq  
			 \exp \left\{ - (1+ n R(d_n)^2/d_n^2)(\log n)^2/\const \right\}.
		\end{equation*}
		Finally for sufficiently large $n$, $h_n(b_n+La_n) \geq R(d_n) - |L/d|$ as $R(x)x^{-d}$ is assumed to be increasing for some $d\in (0,1)$.
		Gathering all the estimates together shows that for $r\leq 2/3$, and sufficiently large $n$,
		\begin{equation*}
			P_n \leq m^{-n} e^{- (\log n)^2/\const}.
		\end{equation*}
		For the case $r >2/3$ we proceed in the same way and get~\eqref{eq:4:Pnbound} with $nK_{\varepsilon_n}(h_n) = (1/2+o(1))  nR(d_n)^2/d_n^2$ and by 
		Lemma~\ref{lem:a:Hbound>23} from the appendix
		\begin{equation*}
			 \E \left[ e^{h_n X} \1_{\left\{ X \in \left(\varepsilon_n, y_n\right] \right\}} \right] \leq   
			 \exp \left\{ \frac{nR(d_n)R'(d_n)}{d_n} \left(\frac{R'(d_n)d_n}{R(d_n)} -1\right)(1+o(1)) \right\}.
		\end{equation*}
		Recall that $b_n = d_n+\tau_n$ with $\tau_n = (1+o(1)) nR'(d_n)/2$ and write
		\begin{equation*}
			h_n(b_n+La_n) = -R(d_n) - (1+o(1)) nR(d_n)R'(d_n)/(2d_n).
		\end{equation*}
		Combining all the estimates yields
		\begin{equation*}
			P_n \leq m^{-n} \exp \left\{ \frac{nR(d_n)R'(d_n)}{d_n} \left(\frac{R'(d_n)d_n}{R(d_n)} -\frac 32 + \frac{R(d_n)}{2R'(d_n)d_n}\right)(1+o(1)) \right\}.
		\end{equation*}
		Since we assume that for some $\epsilon>0$ the functions $R(x)x^{-1/2-\epsilon}$ and $R(x) x^{-1+\epsilon}$ are eventually increasing and 
		decreasing respectively, $y=xR'(x)/R(x) \in [1/2+\varepsilon, 1-\varepsilon]$ 
		for sufficiently large $x$ and thus the expression $y-3/2+1/(2y)$ remains negative and bounded away from $0$.
		We can therefore conclude that for $r >2/3$,
		\begin{equation*}
			P_n \leq m^{-n} \exp \left\{ -nR(d_n)R'(d_n) /(\const d_n)\right\}.
		\end{equation*}
		Note that the sequence $nR(d_n)R'(d_n)/d_n$ is regularly varying with index $3-2/r>0$ and thus it diverges to $+\infty$ as $n \to \infty$.
		\medskip

		Finally to conclude our claim note that we have proven that for any $r \in (0,1)$, and $L\in \R$, 
		\begin{equation*}
			\P\left[ \Lambda_n^{\mathcal{C}} [L, \infty) \geq 1\right] \leq \E\left[ \Lambda_n^{\mathcal{C}} [L, \infty)  \right] \leq \const \: n m^n P_n \leq e^{-  (\log n)^2/\const} \to 0.
		\end{equation*}
	\end{proof}

		Our considerations up to this point show that the particles from $\mathcal{A}_n\cup \mathcal{B}_n\cup \mathcal{C}_n$ do not affect the 
		asymptotic behaviour of  $\scale_{a_n^{-1}}\shift_{-b_n}\Lambda_n$ provided that the 
		latter has a non-trivial limit. In the rest of the article we will show that for 
		$\mathcal{D}_n = \mathbb{T}_n \setminus (\mathcal{A}_n \cup \mathcal{B}_n \cup \mathcal{C}_n)$ the point process
		\begin{equation*}
			\Lambda_n^{\mathcal{D}}=\sum_{w \in \mathcal{D}_n} \delta_{ a_n^{-1}(V(w) - b_n)}
		\end{equation*}
		has a non-trivial limit in distribution. For this reason we need to show that the extremes on $\mathcal{D}_n$ decouple.

\section{Decoupling}\label{sec:dec}

		Until now we investigated only the contribution of 
		big jumps which as it will soon turn out is not enough for our needs. Now we will be also investigating the contribution coming from typical particles. 
		This is when the precise large deviations beyond the so called one big jump domain 
		come into play. In this Section we will state the large deviation result in Theorem~\ref{thm:5:LD} and then use it to show that the extremal positions
		of particles from $\mathcal{D}_n$ decouple in Lemma~\ref{lem:4:Dn}.  \medskip
	
		For functions $f_n, g_n \colon \R \to \R$ we will say 
		that $f_{n-j_n}(z_n) \sim g_{n-j_n}(z_n)$ uniformly with respect to $j_n \in J_n\subseteq \R$ and $z_n \in H_n\subseteq \R$ if
		\begin{equation*}
			\lim_{n \to \infty}\sup_{j_n\in J_n}\sup_{z_n \in G_n} \left| \frac{f_{n-j_n}(x_n)}{g_{n-j_n}(x_n)}-1 \right| = 0.
		\end{equation*}
		Equivalently, $f_{n-j_n}(z_n) \sim g_{n-j_n}(z_n)$ for any choice of $x_n \in H_n$ and $j_n \in J_n$.
		Put 
		\begin{equation}\label{eq:5:hn}
			H_n = \left[d_n+nR'(d_n)/2 -z_n^*, d_n+nR'(d_n)/2+z_n^*\right],
		\end{equation}
		where $z_n^*$ is any sequence of real numbers that is regularly varying with index strictly less than $2-1/r$ if $r > 2/3$ and less than 
		$1/r-1/2$ if $r \leq 2/3$.
		Finally for $x_n \in H_n$ we write
		\begin{equation}\label{eq:5:rate}
			I(x_n) =  \inf_{s \in [0,1]} \left\{ R(x_n-nK'(s)) +  n(sK'(s) - K(s)) \right\}.
		\end{equation}
		Theorem~\ref{thm:5:LD} given below states that the probability that $S_n$ exceeds $x_n\in H_n$ decays as $ne^{-I(x_n)}$. The form of $I(x_n)$ can be 
		heuristically derived in the following way. 
		Write $S_n = N_n+S_n-N_n$. Then $S_n-N_n$ can be seen as the sum of $n-1$ of the $X_k's$ that are not the maximal among $X_1, \ldots , X_n$. 
		In other words. $S_n-N_n$ is the sum of typical values $X_1, \ldots X_n$ and $N_n$ 
		is the one atypical value among $X_1, \ldots , X_n$. Now the event $\{ S_n>x_n\}$ is a union of $\{ N_n> a \} \cap \{ S_n-N_n>b\}$ along $a+b = x_n$. 
		For technical reasons it is most convenient to write $a = x_n -nK'(s)$ and 
		$b = nK'(s)$ for some $s>0$. Then the probability of $\{ N_n = x_n-nK'(s)\}$ decays as $n e^{-R(x_n-nK'(s))}$ and $\{S_n-N_n > nK'(s)\}$ decays as 
		$n^{-1/2}e^{n(K(s)-sK'(s))}$ (see~\eqref{eq:2:rao}). Finally the infimum
		present in the definition of $I(x_n)$ should be interpreted as taking the most probable choice of $s$ and thus of 
		$a = x_n -nK'(s)$ and $b = nK'(s)$.\medskip

 		Note that the infimum in~\eqref{eq:5:rate} is attained for $s=s^*_n$ which is the unique solution to
		\begin{equation*}\label{eq:sstar}
			s^*_n = R'(x_n - nK'(s_n^*))
		\end{equation*}
		in the interval $[0,1]$. By computing the derivative one can easily check that, for a sufficiently large $n \in \N$, 
		the function $s \mapsto R'(x_n - nK'(s))$ is a contraction on $[0,1]$ and thus the existence and 
		uniqueness of $s^*_n$ follows from the Banach fixed 
		point theorem. We thus have the following representation
		\begin{equation*}
			I(x_n) =R(x_n-nK'(s^*_n)) +  n(sK'(s^*_n) - K(s^*_n)) .
		\end{equation*}
		One then checks, using the regular variation of $R'$ that $s^*_n$ is bounded from above by $(1+o(1)) R'(d_n)$.  It transpires that 
		if $ r <2/3$ then $I(x_n) =R(x_n) +o(1)$ uniformly in $x_n\in H_n$ and for $r \geq 2/3$ we have $I(x_n) = R(d_n) + O(n^{4-3/r}\ell(n))$ 
		uniformly in $x_n \in H_n$ for some slowly varying function $\ell$.

		\begin{theorem}\label{thm:5:LD}
			Let Assumption~\ref{as:RW} be satisfied and let $J_n =  \left[0, j_n^* \right] \cap \N$ where $j_n^*$ is any sequence of natural numbers 
			which is regularly varying with index 
			strictly less than $\min \{ 1, 3/r-3\}$. 
			Then uniformly in $x_n \in H_n$ and $j \in J_n$,
			\begin{equation*}
				\P[S_{n-j_n}>x_n] \sim ne^{-I(x_n) + O(j_n^* R'(x_n)^2)}.
			\end{equation*}
		\end{theorem}
		
		The case $r \leq  2/3$ corresponds to~\cite[Theorem 3]{1969:nagaev:integral}, the case $r= 2/3$ to~\cite[Theorem 5]{1969:nagaev:integral} and 
		the case $r > 2/3$ to ~\cite[Theorem 2]{1969:nagaev:integral}.  
		However the statement of the last one was too implicit for our needs and therefore we  provide a new proof with a more explicit statement. 
		We will present a self-contained proof of Theorem~\ref{thm:5:LD} in Section~\ref{sec:LD}. 
		Note that the case $r<2/3$ is covered in~\cite{Denisov2008}. \medskip

		We will use Theorem~\ref{thm:5:LD} in the proof of our decoupling lemma which we will state below. For $k<j$ and $w \in \mathbb{T}$ define
		\begin{equation*}
			V_{k,j}(w) =  \sum_{u\leq w, \: k < |u|\leq j} X_u .
		\end{equation*}
		The next lemma asserts that the particles in $\mathcal{D}_n$ whose positions are around $b_n$ are asymptotically independent.
		The idea of the proof is the following. First, note that for $w \in \mathbb{T}_n \setminus (\mathcal{A}_n \cup \mathcal{B}_n \cup \mathcal{C}_n)$ 
		we have $X_{v^*(w)} > y_n$. One can use this information to infer estimates on $V_{k_n^{(1)},n}(w)$ for sufficiently big $k_n^{(1)}$. 
		Then, using the branching property and the aforementioned estimates, one can get bounds on $V_{k_n^{(2)},n}(w)$ for $k_n^{(2)} < k_n^{(1)}$. One can then
		iterate this procedure $j$ times until $k_n^{(j)}$ gets small enough.

		\begin{lemma}[Decoupling lemma]\label{lem:4:Dn}
			Suppose that Assumptions~\ref{as:BP} and \ref{as:RW} are satisfied. Then there exists a sequence of natural numbers $k_n$ 
			such that $k_n=o\left( n^{1-r-\epsilon}\right)$ for some $\epsilon>0$ 
			and a sequence $l_n=o(b_n)$ such that $a_n=o(l_n)$ with the following property.
			The~sequence of events $\{A_n\}_{n \in \N}$ such that $A_n$ occurs whenever there exist $w, z \in \mathcal{D}_n$ such that 
			\begin{equation*}
				\min\{V_{k_n,n}(w), V_{k_n,n}(z)\}>b_n-l_n, \: v^*(w)\neq v^*(z),\: |w\wedge z| > k_n
			\end{equation*}
			or if $r>2/3$ and there exists $w\in \mathcal{D}_n$ such that $V_{k_n,n}(v^*(w)) > b_n-l_n$ and $|v^*(w)| \leq n-k_n$ 
			satisfies
			\begin{equation*}
				\lim_{n \to \infty} \P \left[ A_n \right] =0.
			\end{equation*}
		\end{lemma}
		\begin{proof}
			We will first treat $r \leq 2/3$. In this case the event $A_n$ is determined only by the first condition. 
			Take $k_n$ to be the smallest integer grater than  
			\begin{equation*}
				(\log n)^3 R'(d_n) (d_n-y_n+1),
			\end{equation*}
			take $l_n=0$ and write, 
			for sufficiently large $n$, 
			a simple estimate
			\begin{equation*}
				\P[A_n] \leq \const m^{2n-k_n}\P[X>y_n]^2\leq e^{-(\log n)^3/ \const} 
			\end{equation*}
			which secures our claim for $r \leq 2/3$. \medskip
			
			The proof for $r> 2/3$ is more involved. In this case the event $A_n$ is determined by both conditions mentioned in the statement.
			For $j \geq 1$ define the sequence of indexes $\{q_j\}_{j \geq 1}$, where $q_1 \in (3-2/r, 1)$ and 
			$q_{j+1}= 1+(q_j-1)/r$. Then $q_{j+1}=1 -2 r^{-j}(1-q_1) \to -\infty$. Let $j_0$ to be the smallest positive integer for which $q_{j_0} < (1-r)$. 
			Note that by the recursive formula for the $q_j$'s,  $q_{j_0} \geq 0$. We can always choose the value of $q_1$ such that the last inequality 
			is strict, so we have $q_{j_0}>0$. 
			Next define the approximating sequences $\{ k_n^{(j)}\}_{n\in \N}$ and $\{ \ell_n^{(j)}\}_{n\in \N}$ via 
			$k_n^{(1)} = n^{q_1}$, $k_n^{(j)} = d_{k_n^{(j-1)}} R'(d_n) \log n$ and 
			$\ell_n^{(j)} = d_{k_n^{(j)}} R'(d_n) (\log n)^{1/2}$. Then for each $j$, $\{ k_n^{(j)}\}_{n\in \N}$ is 
			regularly varying with index $q_j$ and $\{ \ell_n^{(j)}\}_{n\in \N}$ is regularly varying with index $q_j/r$.
			In what follows we will make use of three asymptotic relations between the two sequences which follow from the construction. 
			Namely for each $j\geq 2$, as $n \to \infty$, 
			\begin{equation}\label{eq:5:asympLem}
				k_n^{(j)} \gg R'(d_n) \ell_n^{(j-1)},  \quad \ell_n^{(j)} \gg d_{k_{n}^{(j)}}, \quad \mbox{and}\quad  k_n^{(1)} \gg nR'(d_n)^2.
			\end{equation}
			Define $l_n = l_n^{(j_0)}$ and $k_n = k_n^{(j_0)}$. 
			Consider the subsets of $\mathbb{T}_n$ given via
			$\mathcal{F}_n^{(0)} = \mathcal{D}_n$ and for $j \geq 1$, let
			\begin{align*}
				\mathcal{G}_n^{(j)} & = \left\{ w \in \mathcal{F}_n^{(j-1)} \: : \: | v^*(w)| \leq n-k_n^{(j)}\right\}, \\
				\mathcal{H}_n^{(j)} & = \left\{ w \in \mathcal{F}_n^{(j-1)} \: : \: | v^*(w)| > n-k_n^{(j)},\: V_{k_n^{(j)},n}((v^*(w)) \leq b_n -\ell_n^{(j)}\right\},\\
				\mathcal{F}_n^{(j)} & = \left\{ w \in \mathcal{F}_n^{(j-1)} \: : \:  | v^*(w)| > n-k_n^{(j)}, \: V_{k_n^{(j)},n}(v^*(w)) > b_n -\ell_n^{(j)}   \right\} . \\
						    & = \left\{ w \in \mathcal{D}_n \: : \:  | v^*(w)| > n-k_n^{(j)}, \: V_{k_n^{(i)},n}(v^*(w)) > b_n -\ell_n^{(i)}, i<j   \right\} .	
			\end{align*}
			We have 
			\begin{equation*}
				\P \left[ \mathcal{G}_n^{(1)} \neq \emptyset \right] \leq \const m^{n-k_n^{(1)}}\P[X>y_n] \leq \exp \left\{ - n^{q_1}/\const \right\}
			\end{equation*}
			and 
			\begin{multline*}
				\P \left[ \exists w \in \mathcal{H}_n^{(1)},  \: V_{k_n,n}(w)>b_n-l_n  \right] 
					\leq \\ \const m^{n}\P[X>y_n] \P\left[S_{k_n^{(1)}-k_n}>l_n^{(1)}(1+o(1))\right] \leq \exp \left\{ - n^{q_1}/\const \right\}.
			\end{multline*}	
			Turning our attention to $ \mathcal{G}_n^{(m)}$ and $\mathcal{H}_n^{(m)}$ with $j_0>m\geq 2$ we use the uniform estimates in 
			Theorem~\ref{thm:5:LD} and the first condition in~\eqref{eq:5:asympLem} to write 
 			\begin{multline*}
				\P\left[ \mathcal{G}_n^{(m)} \neq \emptyset \right]  \\
				\leq \P\left[ \exists w \in \mathbb{T}, \:  n-k_n^{(m-1)}< |w| \leq n-k_n^{(m)} , \: V_{k_n,n}(w) >b_n-l_n^{(m-1)} \right]\\
				 \leq \sum_{j=k_{n}^{(m)}}^{k^{(m-1)}_n} m^{n-j} \P[S_{n-j-k_n}>b_n -l_n^{(m-1)}] \leq \exp \left\{- k_n^{(m)}/\const \right\}.
			\end{multline*}
			In a similar way, using the first two conditions in~~\eqref{eq:5:asympLem} and Theorem~\ref{thm:5:LD},
			\begin{multline*}
				\P \left[ \exists w \in\mathcal{H}_n^{(m)} , \: V_{k_n,n}(w)>b_n -l_n \right]\leq  \\
				\leq m^n\sum_{j=0}^{k_n^{(m)}} \P\left[ S_{n-j-k_n} > b_n+La_n-l_n^{(m-1)}\right] \P_{\varepsilon_n}\left[S_j> l^{(m)}_n \right] \\\leq  \exp \left\{- k_n^{(m)} /\const\right\}.
			\end{multline*}	
			From the above considerations
			\begin{equation}\label{eq:5:secondcond}
				\P \left[ \exists w \in \bigcup_{j=0}^{j_0-1} \mathcal{F}_n^{(j)}, \: V_{k_n,n}(v^*(w)) >b_n-l_n, |v^*(w)|\leq n-k_m \right] \to 0.
			\end{equation}
			Next consider the events
			\begin{multline*}
				A_n^{(j)} = \left\{ \exists w, \: z \in \mathcal{F}_n^{(j_0)}\, :  \min \{ V_{k_n,n}(w), V_{k_n,n}(z) \} >b_n-l_n, \right.\\ 
					\left. v^*(w) \neq v^*(z), \: |v^*(w) \wedge v^*(z) | \geq k_n^{(j)} \right\}.
			\end{multline*}
			Then for $j=1$, and sufficiently large $n$,
			\begin{equation*}
				\P\left[  A_n^{(1)} \right] \leq \const m^n m^{n-k_n^{(1)}} \P[X>y_n]^2 \leq \exp \left\{ - n^{q_1} / \const \right\}
			\end{equation*}
			and  for $j>1$,
			\begin{equation}\label{eq:4:aj}
				\P\left[ A_n^{(j)} \right]  \leq  \P\left[ A_n^{(j-1)} \right]  +  \P\left[A_n^{(j)}\setminus A_n^{(j-1)}\right].
			\end{equation}
			On the event $\{k_n^{(j-1)}>|v^*(w) \wedge v^*(z) |\}$ the random variables $V_{k_n^{(j-1)}, n}(w)$ and $V_{k_n^{(j-1)}, n}(z)$ are independent 
			given $\sigma \left( X_v \: : |v|\leq k_n^{(j-1)} \right)$. 
			Now by the definition of $\mathcal{F}_n^{(j-1)}$ we use the aforementioned conditional independence and write
			\begin{multline*}
				\P\left[  A_n^{(j)}\setminus A_n^{(j-1)}\right]\leq
				\P\left[  \exists w, \: z \in \mathcal{F}_n^{(j_0)}, \: V_{k_n,n}(w)>b_n-l_n, \right. \\  V_{k_n^{j-1},n}(z^*(y)) >b_n-l_n^{(j-1)} \left. k_n^{(j-1)}>|v^*(w) \wedge v^*(z) | \geq k_n^{(j)} \right] \\
				\leq \const m^n m^{n-k_n^{(j)}} \P\left[S_{n-k_n^{(j-1)} -k_n} > b_n-l_n^{(j-1)}\right] \P[S_{n-k_n}>b_n-l_n].
			\end{multline*}
			The large deviation result for stretched exponential random variables in Theorem~\ref{thm:5:LD} implies that for sufficiently large $n$, 
			\begin{equation*}
				\P[S_{n-k_n^{(j-1)}} > b_n+La_n-l_n^{(j-1)}] \leq m^n \exp \left\{  \const R'(d_n) l_n^{(j-1)} \right\}.
			\end{equation*}
			By our choice of sequences,  $R'(d_n)l_n^{(j-1)} \ll  k_n^{(j)}$ as stated in~\eqref{eq:5:asympLem}. Going back to~\eqref{eq:4:aj} gives
			\begin{equation*}
				\P\left[ A_n^{(j)} \right]  \leq  \P\left[ A_n^{(j-1)} \right]   + \exp \left\{ - k_n^{(j)}/\const \right\}
			\end{equation*}
			which by induction implies that for any $j \leq j_0$, $\P[A_n^{(j)}]\to0$. The claim follows since $A_n$ is a union of $A_{n}^{(j_0)}$ and the 
			event considered in~\eqref{eq:5:secondcond}.
		\end{proof}
		
		Lemma~\ref{lem:4:Dn} will allow for an efficient decoupling. The position of each $w \in \mathcal{D}_n $ can be written in the following way:
		\begin{equation}\label{eq:4:Vdec}
			V(w) = H_1(w) + T(w) +X_{v^*(w)} + H_2(w),
		\end{equation}
		where
		\begin{align*}
			&H_1(w)   = \sum_{z\leq w} X_z \1_{\{|z|\leq k_n  \}},  
			 &H_2(w)  = \sum_{ z \leq w} X_z\1_{\{ X_z \leq \varepsilon_n, \: |z|>n-k_n \}}, \\ 
			&T(w) =  \sum_{|z|\leq w} X_z\1_{\{ |z|\in [k_n, n-k_n]\}}. 
		\end{align*}

		As one may expect, the first term on the r.h.s. of~\eqref{eq:4:Vdec} is negligible.
		
		\begin{lemma}\label{lem:5:neg}
			For any $r \in (0,1)$,  $\P$ - a.s.
			\begin{equation*}
				\lim_{n \to \infty}\max_{w \in \mathcal{D}_n}H_1(w)/a_n= 0.
			\end{equation*}
		\end{lemma}
		
		\begin{proof}
			One can use the quite generous upper bound 
			\begin{equation*}
				\max_{w \in \mathcal{D}_n}|H_1(w) |\leq M_{k_n}.
			\end{equation*}
			Since $d_{k_n} = o\left( n^{1/r-1-\epsilon} \right)$ for some $\epsilon>0$ as in the claim of Lemma~\ref{lem:4:Dn}, we get $M_{k_n} / a_n \to 0$ 
			as $a_n$ is regularly varying with index $1/r-1$. 
		\end{proof}
		
		As a consequence of Lemma~\ref{lem:5:neg} in order to establish convergence of $\Lambda_n^{\mathcal{D}}$ 
		it is sufficient to study
		\begin{equation*}
				\Theta_n=\sum_{w \in \mathcal{D}_n} \delta_{  a_n^{-1}(T(w)+X_{v^*(w)} +H_2(w) - b_n )}.
		\end{equation*}
		
		\begin{lemma}\label{lem:5:toTheta}
			Suppose that Assumptions~\ref{as:BP} and \ref{as:RW} are satisfied. Let $a_n$ and $b_n$ be chosen as in the statement of 
			Theorem~\ref{thm:2:main}. Then as $n \to \infty$,
			\begin{equation*}
				 \Lambda^{\mathcal{D}}_n-  \Theta_n \Rightarrow o
			\end{equation*}
			in $\Mcal_p$. 
		\end{lemma}
		We will present the proof of this lemma once we establish the convergence of $\Theta_n$.

\section{Convergence of the point processes}\label{sec:lim}
		
	We now turn to the final arguments in in the proof of Theorem~\ref{thm:2:main}. We will first show that the decoupled point process $\Theta_n$ converges to 
	the desired limit. Then, using the tightness of $\Theta_n$, we will show Lemma~\ref{lem:5:toTheta} which will yield immediately Theorem~\ref{thm:2:main}.
	Note first that we can rewrite $\Theta_n$ in the following way
	\begin{equation*}
		\Theta_n = \sum_{n-k_n<|v|\leq n}\1_{\{X_v >y_n\}}  \sum_{z\geq v, \: |z|=n} \delta_{a_n^{-1}(T(z)+X_v+H_2(z)-b_n)}  \1_{E_n(v,z)},
	\end{equation*}
	where $E_n(v,z) = \left\{ X_w  \leq \varepsilon_n \: \mbox{ for all } w \in ( \emptyset, z]\setminus \{v\}, |w|\geq k_n\right\}$.

	\begin{proposition}\label{prop:6:main}
		Suppose that Assumptions~\ref{as:BP} and \ref{as:RW} are satisfied. Then for sequences $\{ b_n\}_{n\in \N}$ and $\{a_n\}_{n \in \N}$ as in the 
		statement of Theorem~\ref{thm:2:main}, and $\Lambda $ given by \eqref{eq:3:LambdaLimit},
		\begin{equation*}
			\Theta_n \Rightarrow \Lambda
		\end{equation*}
		in $\Mcal_p$.
	\end{proposition}
	
	\begin{proof}
		By the choice of the centring sequence $\{ b_n\}_{n \in \N}$ and the scaling sequence $\{a_n\}_{n \in \N}$, by an appeal to Lemma~\ref{lem:3:an}, 
		Lemma~\ref{lem:3:cr<1} and Theorem~\ref{thm:5:LD} we have for any $x \in \R$,  
		\begin{equation*}
			\lim_{n \to \infty }m^n\P\left[  S_{n-2k_n}+X >b_n + x a_n , \:  N_{n-2k_n} \leq \varepsilon_n, \: X>y_n \right] = ae^{-x}.
		\end{equation*}
		Therefore if we denote
		\begin{align*}
			\eta_n(\cdot) & = \P\left[  a_n^{-1}(S_{n-2k_n}+X-b_n) \in \cdot \:, \:  N_{n-2k_n} \leq \varepsilon_n, \: X>y_n\right] 
		\end{align*}
		the above observation implies that vaguely
		\begin{equation*}
			m^n \eta_n(\ud x) \to \eta(\ud x) = ae^{-x}\ud x.
		\end{equation*}
		Fix any $f \in \Ccal^+_K$ and take $n\in \N$ sufficiently small such that ${\rm supp}(f) \subseteq [-l_n/a_n, \infty]$. 
		Recall the statement of Lemma~\ref{lem:4:Dn} and the definition of the event $A_n$. By conditioning on 
		$\mathbb{T}$ and  $\mathcal{E}_n=\sigma(X_u, {|u| \geq n-k_n})$ we can write on the set $A_n^c$, 
		\begin{multline*}
			\E \left[\left. \exp \left\{ - \int f(x) \: \Theta_n (\ud x) \right\} \right| \mathbb{T}, \mathcal{E}_n \right]= \\
				\prod_{|v|\in(n-k_n, n] }\!\!\! \E \left[ \left.
					\exp\left\{ - \1_{E_n} \!\!\!\sum_{w\geq v, |w|=n}\!\!\!\1_{E_n'(v,w)}f \left(\frac{S_{n-k_n-1} + H_2(w) +X_v}{a_n} \right)  \right\} \right| 
				\mathbb{T}, \mathcal{E}_n  \right]
		\end{multline*}		
		where $S_{n-k_n-1}$ is independent from everything else,
		\begin{equation*}
			E_n'(v,w) = \left\{ X_z  \leq \varepsilon_n \: \mbox{ for all } z \in ( \emptyset, w]\setminus \{v\}, |z|\geq n-k_n, X_v>y_n\right\} \in \mathcal{E}_n
		\end{equation*}
		and
		$E_n= \left\{  N_{n-k_n-1} \leq \varepsilon_n \right\}$.
		Integrating the conditional expectation over the set $A_n^c$ and using the fact that $\P[A_n]=o(1)$ we get
		\begin{equation*}	
			\E \left[\exp \left\{ - \int f(x) \: \Theta_n (\ud x) \right\} \right] 
		= \E\left[\prod_{k=n-k_n}^n  \E \left[ \exp \left\{ - \int f(x) \: \Theta_{n,k} (\ud x) \right\} \right] ^{Z_k} \right]+o(1)
		\end{equation*}
		where
		\begin{equation*}
			\Theta_{n,k} =  \1_{E_n} \1_{\{ X >y_n \}}\sum_{ w \in \mathcal{A}_{n-k}}   \delta_{a_n^{-1}\left( S_{n-2k_n-1} +X +V(w)-b_n\right)},
		\end{equation*}
		We will argue that for $ j \in [0, k_n)$, 
		\begin{multline}\label{eq:6:claim}
			 Z_{n-j}\E \left[  1-  \exp \left\{ - \int f(x) \: \Theta_{n,n-j} (\ud x) \right\} \right] \\ \to m^{-j}W \int \E \left[ 1- e^{-Z_{j}f(x)}  \right] \eta(\ud x)
		\end{multline}
		a.s. which implies 
		\begin{multline*}
			\prod_{j=0}^{k_n}  \E^* \left[   \exp \left\{ - \int f(x) \: \Theta_{n,n-j} (\ud x) \right\} \right] ^{Z_{n-j}} \to \\ \E^* \left[ \exp \left\{ - W \sum_{j=0}^\infty m^{-j}  \int \E \left[ 1- e^{-Z_jf(x)}  \right] \eta(\ud x)  \right\}\right].
		\end{multline*}
		To show~\eqref{eq:6:claim} define the modulus of continuity of $f$ via
		\begin{equation}\label{eq:6:modulusf}
				\omega_f (\epsilon) = \sup_{|x-y| \leq \epsilon} |f(x)-f(y)|, \quad \epsilon>0\, .
		\end{equation}
		Denote
		\begin{equation*}
			\Delta_{n-j}= \int f(x) \: \Theta_{n,n-j} (\ud x) - Z_j f \left( a_n^{-1}\left( S_{n-2k_n} +X -b_n\right)\right)
		\end{equation*}
		and notice that the law $ |\Delta_{n-j}|$ is dominated in the sense of the stochastic order via
		\begin{equation*}	
			 |\Delta_{n-j}|\stackrel{d}{\leq} Z_j \omega_f\left( M_{k_n}/a_n  \right)
		\end{equation*}
		where the second term on the right hand side goes to $0$ in probability as $n \to \infty$, since $M_{k_n}/a_n  \overset{\P}\to 0$ (see the proof of Lemma~\ref{lem:5:neg}) . Note that since $f$ is necessarily bounded, the aforementioned convergence holds also in $L^1$. Now write
		\begin{multline*}
					 Z_{n-j}  \E \left[  1-  e^{ - \int f(x) \: \Theta_{n,n-j} (\ud x)} \right] 
					 	= Z_{n-j} \E \left[  1-  e^{-Z_j f \left( a_n^{-1}\left( S_{n-2k_n} +X -b_n\right)\right)} \right] \\
					 	+ Z_{n-j}\E \left[e^{-Z_j f \left( a_n^{-1}\left( S_{n-2k_n} +X -b_n\right)\right)} (1-e^{-\Delta_{n-j}})  \right].
		\end{multline*}
		The first term has the desired asymptotics, i.e. by an appeal to the vague convergence $m^n \eta_n(\cdot) \to \eta(\cdot)$, we can write
		\begin{multline*}
			Z_{n-j} \E \left[  1-  e^{-Z_j f \left( a_n^{-1}\left( S_{n-2k_n} +X -b_n\right)\right)} \right]   =  Z_{n-j}\int \E \left[1-e^{-Z_jf(x)} \right] \eta_n(\ud x) \\
				 = \frac{Z_{n-j}}{m^{n-j}} m^{-j}\int \E \left[1-e^{-Z_jf(x)} \right]  \: m^n\eta_n(\ud x) \rightarrow  m^{-j}W \int \E \left[ 1- e^{-Z_{j}f(x)}  \right] \eta(\ud x).
		\end{multline*}
		The second term is negligible since
		\begin{multline*}
			Z_{n-j}\E \left[e^{-Z_j f \left( a_n^{-1}\left( S_{n-2k_n} +X -b_n\right)\right)} (1-e^{-\Delta_{n-j}})  \right] \\
				\leq Z_{n-j}\E \left[e^{-Z_j f \left( a_n^{-1}\left( S_{n-2k_n} +X -b_n\right)\right)} \Delta_{n-j}  \right] \\
			 \leq \frac{Z_{n-j}}{m^{n-j}}  \E \left[ \frac{Z_j}{m^j} \omega_f\left( \frac{M_{k_n}}{R'(d_n)} \right)  \right] m^n \eta_n( (\varepsilon, \infty]).
		\end{multline*}
		The convergence in ~\eqref{eq:6:claim} follows and so does the claim.
		\end{proof}

		\begin{proof}[Proof of Lemma~\ref{lem:5:toTheta}]
			Fix $f \in \Ccal_K^+$ and recall \eqref{eq:6:modulusf}. 
			As an element of $\Ccal_K^+$ the function $f$ must have a limit at $+\infty$ and thus  $\omega_f(\epsilon) \to 0$ as $\epsilon \to 0$. 
			For any fixed $\epsilon>0$, on the event $B_n(\epsilon)=\left\{\max_{w \in \mathcal{D}_n}|H_1(w) |< \epsilon /R'(d_n) \right\}$,
			\begin{align*}
				&\left| \int f(x) \: \Lambda^{\mathcal{D}}_n(\ud x) - \int f(x) \: \Theta_n (\ud x) \right|  \leq  \Theta_n  ((\epsilon, \epsilon)+{\rm supp}(f)) \omega_f(\epsilon).
			\end{align*}
			The claim follows since the events $B_n(\epsilon)$ have probability close to $1$ and $\omega_f(\epsilon) \to 0$ while the measure of 
			$(\epsilon, \epsilon)+{\rm supp}(f)$ with respect to $ \Theta_n$ 
			remains bounded in probability as $\epsilon \to 0$.
		\end{proof}

	\begin{proof}[Proof of Theorem~\ref{thm:2:main}]
		The claim is an immediate consequence of Lemma~\ref{lem:5:toTheta} and Proposition~\ref{prop:6:main}.
	\end{proof}

\section{Large deviations for stretched exponential tails}\label{sec:LD}

		In this section we will revisit precise large deviations for sums of i.i.d. random variables with stretched exponential tails. This problem was addressed previously, see~\cite{1969:nagaev:integral, 1969:nagaev:integral2}. However the general
		results in~~\cite{1969:nagaev:integral} were not sufficiently explicit for our purposes. For this reason we decided to revisit the problem and study the deviations on the scale $d_n$, where $d_n$ is any regularly varying 
		function with index $1/r$. 
		We will study deviations uniformly in the time interval
		\begin{equation*}
			n_0 = n-j_n, \quad j_n \in J_n =  \left[0, j_n^* \right],
		\end{equation*}
		where $j_n^*$ is any sequence of natural number which is regularly varying with index strictly less than $\min \{ 1, 3/r-3\}$. 
		Recall $H_n \subseteq \R$ given in~\eqref{eq:5:hn} and the rate function $I(x_n)$ given in~\eqref{eq:5:rate}. 

		The first lemma asserts that whenever $S_{n_0}>x_n$ there has to be at least one jump grater than $\varepsilon_n= \varepsilon d_n$ for some fixed $\varepsilon \in (2^{-r},1)$. 
		The proof goes along the exact same lines as the proof of Lemma~\ref{lem:3:an} (see~\eqref{eq:4:dtrongerclaim}) 
		and is therefore omitted.

	\begin{lemma}
		Suppose that Assumption~\ref{as:RW} is satisfied. Then  
		\begin{equation*}
			\sup_{x_n \in H_n} \sup_{j_n \in J_n}\P[S_{n_0}>x_n, \: N_{n_0}\leq \varepsilon_n] =o \left( ne^{-I(x_n)} \right) .
		\end{equation*}
	\end{lemma}
	
	Next we will treat the possibility of having two jumps bigger than $\varepsilon_n$. The proof is straightforward, see~\eqref{eq:4:bn}, and is therefore also omitted.

	\begin{lemma}
		Suppose that Assumption~\ref{as:RW} is satisfied. Then
		\begin{equation*}
			\sup_{x_n \in H_n} \sup_{j_n \in J_n}\P[S_{n_0}>x_n, \: \exists i, j\leq n_0, \:  \min \{ X_i, X_j\} >\varepsilon_n, \: i\neq j] = o \left( n e^{-I(x_n)} \right).
		\end{equation*}
	\end{lemma}
	Taking the last two lemmas into account, we see that in order to establish the asymptotics of $\P[S_{n_0}>x_n]$ it is sufficient to study 
 	$$\P\left[ S_{n_0} +X > x_n , \: N_{n_0} \leq \varepsilon_n, \: X > \varepsilon_n \right]. $$
	In our next step we will show that $X$ has to be sufficiently big. Recall the sequence $\{ y_n\}_{n \in \N}$ given by~\eqref{eq:4:yn}.
	
	\begin{lemma}
		Suppose that Assumption~\ref{as:RW} is satisfied. Then 
		\begin{equation*}
			\sup_{x_n \in I_n} \sup_{j_n \in J_n}\P\left[ S_{n_0} +X > x_n , \: N_{n_0} \leq \varepsilon_n, \: X \in \left(\varepsilon_n, y_n\right] \right] = o \left( n e^{I(x_n)} \right).
		\end{equation*} 
	\end{lemma}
	The proof goes along the same lines as the proof of Lemma~\ref{lem:3:cr<1} and is therefore omitted. 	
	In what follows we will investigate
	\begin{equation*}
		P_n  = \P\left[ S_{n_0} +X > x_n , \: N_{n_0} \leq \varepsilon_n, \: X > y_n\right] .
	\end{equation*}
	Consider $w_n>y_n$ given via
	\begin{equation}\label{eq:4:ynNEW}
		w_n=\left\{ \begin{array}{rr}
				d_n +  (\log n)^2/R'(d_n) & r\in \left(0,  2/3 \right] \\
				d_n + n R'(d_n)/3& r \in \left(  2/3, 1 \right) 		
			\end{array}\right..
	\end{equation}

 	We will start by writing 
	\begin{multline*}
		P_n =\P\left[ S_{n_0} > x_n-y_n , \: N_{n_0} \leq \varepsilon_n, \: X > y_n\right] \\
		+ \P\left[ S_{n_0} < x_n-w_n , \: N_{n_0} \leq \varepsilon_n, \: S_{n_0}+X>x_n, \: X>y_n\right]  \\+\P\left[ x_n-S_{n_0} \in[ y_n, w_n], \: N_{n_0} \leq \varepsilon_n, \: X > x_n-S_{n_0}\right] = A_n+B_n+C_n.
	\end{multline*}

	\begin{lemma}
		Suppose that Assumption~\ref{as:RW} is satisfied. We have 
		\begin{equation*}
			\sup_{x_n \in I_n} \sup_{j_n \in J_n}(A_n+B_n) = o \left( n e^{-I(x_n)} \right).
		\end{equation*}
	\end{lemma}

	\begin{proof}
		We will first treat $A_n$. Simply write
		\begin{equation*}
			A_n \leq \P[X>y_n] \P_{\varepsilon_n}[S_{n_0} > x_n-y_n].
		\end{equation*}
		If $r \leq 2/3$ the above estimate combined with an exponential Markov's inequality with $s = (1+\epsilon)R'(d_n)$ constitutes
		\begin{equation*}
			A_n \leq \P[X >y_n] \exp \{ n_0K_{\varepsilon_n}(s) -s(x_n-y_n)  \}. 
		\end{equation*}
		Provided that $\epsilon>0$ is sufficiently small, we can invoke Lemma~\ref{lem:4:asK} and conclude that
		\begin{equation*}
			A_n \leq \exp \left\{-R(d_n)- \epsilon (\log n)^2 \right\} .
		\end{equation*}
		For $r >2/3$ we apply the same argument (with the same choice of $s$ in the exponential Markov's inequality) and get
		\begin{equation*}
			A_n \leq  \exp \left\{-R(d_n)- \epsilon n R'(d_n)^2\right\} .
		\end{equation*}
		This proves our first claim.
		For the second term we have a simple bound $B_n \leq \P[X>w_n]$ which implies our claim.
	\end{proof}

	To treat $C_n$ firstly condition on $X_1, \ldots X_{n_0}$ and write
	\begin{equation}\label{eq:7:Cn}
		C_n = \E[e^{-R(x_n-S_{n_0})} \1_{\{ N_{n_0} \leq \varepsilon_n, \:  x_n-S_{n_0} \in[ y_n, w_n] \}}].
	\end{equation}
	Recall the measure $\P_{\varepsilon_n}$  introduced in Section~\ref{sec:trimming} and that we write $\E_{\varepsilon_n}$ for 
	its expectation.
 	We have uniformly in $x_n \in H_n$ and $j_n \in J_n$, 
	\begin{equation*}
		C_n =(1+o(1))  \E_{\varepsilon_n}[e^{-R(x_n-S_{n_0})} \1_{\{  x_n-S_{n_0} \in[ y_n, w_n] \}}].
	\end{equation*}	
	We will use yet another change of measure. 
	For $s>0$ denote by $\P_{s,\varepsilon_n}$ a probability measure with respect to which $X_k$ are i.i.d. distributed according to  
	\begin{equation*}
		\P_{s,\varepsilon_n}[X_k \leq t  ] = e^{-K_{\varepsilon_n}(s)} \E_{\varepsilon_n} \left[ \1_{\{ X_k\leq t\}} e^{sX_k} \right].
	\end{equation*}	
	Denoting by $\E_{s,\varepsilon_n}$ the expected value corresponding to $\P_{s,\varepsilon_n}$ we have $\E_{s,\varepsilon_n}[X_1] = K'_{\varepsilon_n}(s)$ and 
	$\E_{s, \varepsilon_n}\left[(X_1 - \E_{s, \varepsilon_n}[X_1])^2 \right] = K''_{\varepsilon_n}(s)$.
	Consider the distribution of normalized sums 
	\begin{equation*}
		F_n^{(s)}(y) = \P_{s, \varepsilon_n} \left[ \theta_n(S_{n_0}) \leq y  \right], \quad \text{where}\quad  \theta_n(y) = (y-n_0K'_{\varepsilon_n}(s))/\sqrt{n_0K''_{\varepsilon_n}(s)}.
	\end{equation*}
	Then $F_n^{(s)}$ converge point wise to the cumulative distribution function of standard normal distribution. 
 	and observe that for $D_n = [\theta_n(x_n-w_n), \theta_n(x_n-y_n)]$ we have
	\begin{multline*}
		 \E_{\varepsilon_n}[e^{-R(x_n-S_{n_0})} \1_{\{  x_n-S_{n_0} \in[ y_n, w_n] \}}] = \\
		e^{n_0 K_{\varepsilon_n}(s)}\E_{s, \varepsilon_n} \left[ e^{ -s S_{n_0}-R(x_n-S_{n_0})} \1_{\{  x_n-S_{n_0} \in[ y_n, w_n] \}}\right] \\
			  =   \exp\left\{ n_0K_{\varepsilon_n}(s)  \right\} \int_{D_n} e^{-s\theta_n^{-1}(y) - R(x_n-\theta_n^{-1}(y))} \: F_n^{(s)}(\ud y).
	\end{multline*}
	Now use Taylor expansion for $R$ to write, for $y \in \theta_n( [x_n-w_n, x_n-y_n])$, and $x_n(s) = x_n-n_0K_{\varepsilon_n}'(s)$, 
	\begin{equation*}
		R(x_n-\theta_n^{-1}(y)) = R(x_n(s)) -yR'(x_n(s)) \sqrt{n_0K''(s)} +  y^2R''(\xi_{y}) n_0K''(s) /2
	\end{equation*}
	for some $\xi_{y}$ that lies between $x_n(s)$ and $x_n(s)- y \sqrt{n_0K''_{\delta}(s)}$. 
	Thus we arrive at
	\begin{equation*}
		C_n =(1+o(1)) e^{ n_0K_{\varepsilon_n}(s) - n_0sK'_{\varepsilon_n}(s) - R(x_n(s)) } \int_{D_n} e^{y e_n(s) -y^2 R''(\xi_y) n_0K''_{\varepsilon_n}(s)/2}  \: F_n^{(s)}(\ud y)
	\end{equation*}
	with
	\begin{equation*}
		e_n(s) = (R'(x_n(s)) - s)\sqrt{n_0K''_{\varepsilon_n}(s)}.
	\end{equation*}
	We want to take $s=s^*$ in the above construction, where $s^*$ is characterised by being the unique solution to~\eqref{eq:sstar}. Then, by the merit of regular variation of $R''$, for some slowly varying function $\ell_n$,
	\begin{equation*}
		|e_n(s^*)| \leq |R'(x_n(s^*)) - R'(x_n-nK'(s^*))| \sqrt{n K''_{\varepsilon_n}(s^*)} \leq n^{-1/2} \ell_n.
	\end{equation*}
	Note that we used here the bound on $j_n^*$. In the sequel we will write $e_n = e_n(s^*)$.
	Finally by the regular variation of $R'''$ for some slowly varying sequence $\ell_n$
	\begin{equation}\label{eq:7:epsilonstar}
		\sup_{z \in [x_n-w_n, x_n-y_n]} |R''(\xi_{x_n, x_n+z}) - R''(x_n)| \leq  \epsilon_n=n^{3-4/r}\ell_n.
	\end{equation}
	It is therefore sufficient to show the claim of our next lemma.

	\begin{lemma}\label{lem:7:last}
		Suppose that Assumption~\ref{as:RW} is satisfied. Let 
		$f_n =(R''(x_n(s^*)) +\epsilon_n) n_0K''_{\varepsilon_n}(s^*)$ where $\{\epsilon_n\}_{n \in \N}$ is regularly varying with index $3-4/r$. 
		Then, as $n \to \infty$, 
		\begin{equation*}
			\sup_{x_n \in I_n} \sup_{j_n \in J_n}  \left| \int_{D_n} e^{y e_n -y^2 f_n/2}  \: F_n^{(s^*)}(\ud y) - 1 \right| \to 0.
		\end{equation*}
	\end{lemma}

	\begin{proof}
		Firstly, for $r \geq 2/3$, $y e_n -y^2 f_n/2\to 0$ uniformly in $y \in D_n$, $x_n \in H_n$, and $j_n \in J_n$ and so the claim follows from the central limit theorem. We thus treat only the case $r >2/3$. 
		We will utilize the following estimate, for $y \in [0, \theta_n(x_n-y_n)]$ and $h>0$  we have
	\begin{multline*}
		1-F_n^{(s^*)} (y)  = e^{-n_0 K'_{\varepsilon_n}(s^*)} \E_{\varepsilon_n}[ e^{s^* S_{n_0}} \1_{\{ \theta_n(S_{n_0}) >y \}}] \leq \\
			(1+o(1))\exp\left\{ n_0(K_{\varepsilon_n}(s^*+h) - K_{\varepsilon_n}(s^*) -hK'_{\varepsilon_n}(s^*)) - \sqrt{n_0 K''_{\varepsilon_n}(s^*)} hy  \right\}.
	\end{multline*}
	Take $h  = \epsilon y/\sqrt{n_0}$ for $\epsilon>0 $ sufficiently small so that we can invoke Lemma~\ref{lem:4:asK} to get
	\begin{equation}\label{eq:7:upper}
		1-F_n^{(s^*)} (y)  \leq \const \exp\left\{ - y^2 / \const_\epsilon  \right\},
	\end{equation}
	for some constant $\const_\epsilon$ that depends on $\epsilon$.
	Similarly, for $y \in [\theta_n(x_n-w_n),0]$
	\begin{equation*}
		F_n^{(s^*)} (y)  \leq \const \exp\left\{ - y^2 / \const_\epsilon  \right\}.
	\end{equation*}

		Denote $D_n^{+} = \left[ 0, \theta_n(x_n-y_n)\right]$.
		We will first show that
		\begin{equation*}
			\sup_{z_n \in I_n} \sup_{j_n \in J_n}  \left| \int_{D_n^+} e^{y e_n -y^2 f_n/2}  \: F_n^{(s^*)}(\ud y)  - \frac 12\right| \to 0.
		\end{equation*}
		Using the integration by parts formula we arrive at
		\begin{multline*}
			 \int_{D_n^+} e^{y e_n -y^2 f_n/2}  \: F_n^{(s^*)}(\ud y)  =  \int_{D_n^+} (e_n -y f_n) e^{y e_n -y^2f_n /2} (1- F_n^{(s^*)}( y) )\: \ud y \\
				- e^{\theta_n(x_n-y_n) e_n -\theta_n(x_n-y_n)^2f_n /2} (1- F_n^{(s^*)}( \theta_n(x_n-y_n))) + (1- F_n^{(s^*)}( 0)) .
		\end{multline*}
		One can show that the first two terms on the right hand side are negligible using the estimate~\eqref{eq:7:upper}. Indeed the behaviour of the second term follows directly from~~\eqref{eq:7:upper}. 
		The integral can be treated by considering $y > n^{ 2/r-2-\epsilon}$ and $y< n^{2/r-2-\epsilon}$. In the first case the 
		pre factor $e_n +f_ny$ is bounded via $n^{-\epsilon}$. In the second case the exponent tends to $-\infty$.
		Our claim follows since $F_n^{(s^*)}( 0) \to 1/2$.
		In a similar fashion we treat $D_n^{-} = \left[\theta_n(x_n-w_n),  0 \right]$, we write
		\begin{multline*}
			 \int_{D_n^-} e^{y e_n -y^2 f_n/2}  \: F_n^{(s^*)}(\ud y)  =  \int_{D_n^-} (e_n -y f_n) e^{y e_n -y^2f_n /2} (1- F_n^{(s^*)}( y) )\: \ud y \\
				- e^{\theta_n(x_n-w_n) e_n -\theta_n(x_n-w_n)^2f_n /2} (F_n^{(s^*)}( \theta_n(x_n-w_n))) + F_n^{(s^*)}( 0) .
		\end{multline*}
		and conclude that the limit of the right hand side is $1/2$. 

\end{proof}

	\begin{proof} [Proof of Theorem~\ref{thm:5:LD}]
		Our considerations up to this point show that it is enough to justify that $C_n$ given in~\eqref{eq:7:Cn} has the desired asymptotics. Note that with $\varepsilon_n^*$ be given 
		via~\eqref{eq:7:epsilonstar} one has
		\begin{multline*}
			 \int_{D_n} e^{y e_n +y^2 f_n^\downarrow/2}  \: F_n^{(s^*)}(\ud y) \leq  \\ C_n  e^{n_0s^*K'_{\varepsilon_n}(s^*)  - n_0K_{\varepsilon_n}(s^*) + R(x_n(s^*)) } \leq \int_{D_n} e^{y e_n +y^2 f^\uparrow_n/2}  \: F_n^{(s^*)}(\ud y) 
		\end{multline*}
		where $f^\downarrow_n =(R''(x_n(s^*)) +\varepsilon_n^*) n_0K''_{\varepsilon_n}(s^*)$ and $f_n^\uparrow =(R''(x_n(s^*)) -\varepsilon_n^*) n_0K''_{\varepsilon_n}(s^*)$. The claim now follows from Lemma~\ref{lem:7:last} and the fact that
		\begin{equation*}
			n_0s^*K'_{\varepsilon_n}(s^*)  - n_0K_{\varepsilon_n}(s^*) + R(x_n(s^*))  = I(x_n) + O\left(j_n^*R'(x_n)^2\right).
		\end{equation*}
	\end{proof}

\appendix

\section{Appendix}

Here we present the proof of the auxiliary lemmas used in the article. On a few occasions will make use of the following integration by parts formula: 
for $A<B$ and a continuously differentiable function $\psi$,
\begin{multline}\label{eq:4:intbyparts}
	\E \left[\psi(X)\1_{\left\{ X \in \left(A, B\right] \right\}} \right] =\\ \int_{A}^{B} \psi'(s) \P[X>s]\ud s + \psi(A)\P\left[X >A\right] - \psi(B)\P[X>B].
\end{multline}

In the proof of Lemma~\ref{lem:4:asK} it will be convenient to first analyse the moment generating function of the truncated random variables
\begin{equation*}
	F_{\varepsilon_n}(s) = \E_{\varepsilon_n} \left[  \exp\{ sX_1\} \right],
\end{equation*}
where $\varepsilon_n=\varepsilon d_n$, with $\varepsilon \in \left( 2^{-r}, 1 \right)$.
We begin with a result which will allow to control $F_{\varepsilon_n}$. Recall that $\kappa$ is the smallest integer greater than or equal to $(2-r)/(1-r)$. Recall that under Assumption~\ref{as:RW} the 
function $R(x)x^{-d}$ is decreasing for some $d \in (0,1)$. 

\begin{lemma}\label{lem:a:Funiform}
	Let Assumption~\ref{as:RW} be in force and let $d\in (0,1)$ be such that $R(x)x^{-d}$ is decreasing. Then for any $k >\kappa$ such that $\E[|X|^k]<\infty$ 
	and $\chi \in (0, 1/d)$ we have, as $n \to \infty$, 
	\begin{equation*}
		\sup_{0 \leq s \leq  \chi R'(\varepsilon_n) } s^{k} F_{\varepsilon_n}^{(k)}(s) = o\left( n^{-1}\right).
	\end{equation*}
\end{lemma}
\begin{proof}
	As $X$ is bounded under $\P_{\varepsilon_n}$,
	\begin{equation*}
		F_{\varepsilon_n}^{(k)}(s) = \E_{\varepsilon_n} \left[ X^{k} e^{s X} \right].
	\end{equation*}
	Since $\P[X \leq \varepsilon_n]\to 1$ it is sufficient to estimate
	\begin{equation*}
		 \E \left[ X^{k} e^{s X} \1_{\{ X \leq \varepsilon_n\}} \right] =  \E \left[ X^{k} e^{s X} \1_{\{ X \leq 0 \}} \right] +  
		 	\E \left[ X^{k} e^{s X} \1_{\{ 0<X \leq \varepsilon_n\}} \right]
	\end{equation*}
	The first integral can be bounded from above via $ \E \left[ |X|^{k} \right]$. We will apply the integration by parts formula~\eqref{eq:4:intbyparts} for the second expectation  with $\psi(x) = x^{k}e^{s x} $, 
	$A =0$ and $B=\varepsilon_n$. Before we present the next bound, let us note that
	\begin{equation}\label{eq:a:unif1}
		\sup_{0 \leq s \leq  \chi R'(\varepsilon_n) } \left [s\varepsilon_n -R(\varepsilon_n)\right] \leq (\chi d -1) R(\varepsilon_n).
	\end{equation}
	With this estimate we infer that for sufficiently large $n$,
	\begin{equation*}
		\psi(\varepsilon_n) \P[X>\varepsilon_n] \leq  e^{- R(d_n)/\const}.
	\end{equation*}
	For the integral we can write
	\begin{multline*}
			  \int_0^{\varepsilon_n }\psi'(t) \P[X>t]\ud t 
			  	 \leq \const\cdot d_n^k \int_{0}^{\varepsilon_n}  \exp \left\{ s t - R(t) \right\} \ud t  \\
				 \leq\const\cdot  d_n^{k+1}  \int_{\delta + o(1)}^1\exp \left\{   s \varepsilon_nt-R(\varepsilon_n t)\right\} \ud t 
		\end{multline*}
		As $R$ is regularly varying with index $r \in (0,1)$, $R(\varepsilon_n t) \sim t^r R(\varepsilon_n)$ almost uniformly in $t \in (0,\infty)$, i.e. uniformly in $t \in G$ for any compact $G \subseteq (0,+\infty)$.
Thus for arbitrarily small $\varepsilon>0$ there is $n$ large enough such that $R(\varepsilon_n t) \geq t^r R(\varepsilon_n)(1-\varepsilon)$ for all $t \in [\delta/2, 1]$. Thus the integral can be bounded in the following way
		\begin{equation*}
			 \int_0^{\varepsilon_n }\psi'(t) \P[X>t]\ud t \leq C d_n^{k+1} \exp \{ (\chi d -1+\varepsilon) R(\varepsilon_n) \}.
		\end{equation*}
		Provided that $\varepsilon$ is sufficiently small the right hand side decays exponentially fast in $n$ which in turn secures our claim.
\end{proof}

We now turn to the proof of Lemma~\ref{lem:4:asK} which we recall here for convenience.

	\begin{lemma}
	Let Assumption~\ref{as:RW} be in force. Let $d\in (0,1)$ be such that $R(x)x^{-d}$ is decreasing and let $\chi \in (0, 1/d)$.
	Then for any $i \in \{0, 1,2\}$, 
		\begin{equation*}
			\sup_{0 \leq s \leq  \chi R'(\varepsilon_n) } \left| K_{\varepsilon_n}^{(i)}(s) - K^{(i)}(s) \right| = o \left( \frac 1n \right).
		\end{equation*}
	\end{lemma}

	\begin{proof}[Proof of Lemma~\ref{lem:4:asK}]
		Since under $\P_{\varepsilon_n}$, $X$ is bounded a.s. for any fixed $s$,
		\begin{align*}
			F_{\varepsilon_n}(s) = \sum_{j=0}^{\kappa} \frac{s^j}{j!} \E_{\varepsilon_n}\left[X^j\right] + \int_0^s \frac{(s-t)^\kappa}{(\kappa+1)!} F_{\varepsilon_n}^{(\kappa+1)}(t) \ud t, 
		\end{align*}
		where for any fixed $j$, $\E_{\varepsilon_n}\left[X^j\right]=\E[X^j]+ o\left(1/n \right)$. Let
		\begin{equation*}
			 F(s) =  \sum_{j=0}^{\kappa} \frac{s^j}{j!} \E\left[X^j\right] .
		\end{equation*}
		By Lemma~\ref{lem:a:Funiform} and the choice of $\kappa$,  for any $i \in \{0, 1,2\}$, 
		\begin{equation}\label{eq:a:lasttt}
			\sup_{0 \leq s \leq  \chi R'(\varepsilon_n) } |  F^{(i)}(s) -  F_{\varepsilon_n}^{(i)}(s)|  = o\left(1/n\right).
		\end{equation}

		We continue with
		\begin{equation}\label{eq:4:KF}
			K_{\varepsilon_n}(s) = \log(F_{\varepsilon_n}(s)) = \log(F(s)) + \widehat{\Delta}_n(s),
		\end{equation}
		with $\widehat{\Delta}_n(s)$ such that $\sup_{0 \leq s \leq  \chi R'(\varepsilon_n) }|\widehat{\Delta}_n(s)|= o\left( 1/n\right)$.
		Define $\phi(s) = \log(F(s))$. Following~\cite{smith:1995} we have $\phi(0)=0$ and
		\begin{equation*}
			\frac{\ud }{\ud s}F(s) = \frac{\ud}{\ud s} e^{\phi(s)} = \phi'(s) e^{\phi(s)} = \phi'(s)F(s).
		\end{equation*}
		We see that $\phi'(0)=0$ and after differentiating $m-1$ times, where $m\leq \kappa$, we get
		\begin{equation*}
			 \frac{\ud^{m}}{\ud s^{m}}\phi(0) = \E[X^m] - \sum_{j=1}^{m-1}  \binom{m -1}{ j}  \frac{\ud^{m-j}}{\ud s^{m-j}}\phi(0) \cdot \frac{\ud^{j}}{\ud s^{j}}F(0).
		\end{equation*}
		Comparing the above relation with~\eqref{eq:4:c1} allows us to conclude that $\frac{d^{m}}{ds^{m}}\phi(0) = k_m$ for $m \leq \kappa$.
		Using the Taylor expansion of $\phi$ we get finally that
		\begin{equation*}
			\log F(s) = \phi(s) =\sum_{j=2}^{\kappa } \frac{k_j}{j!}s^j + \frac{s^{\kappa+1}}{(\kappa+1)!}\phi^{(\kappa+1)}(\vartheta s)  = K(s) + \frac{s^{\kappa+1}}{(\kappa+1)}\phi^{(\kappa+1)}(\vartheta s)
		\end{equation*}
		for some $\vartheta \in [0,1]$ with 
		\begin{equation*}
			\sup_{0 \leq s \leq \chi R'(\varepsilon_n)  } \left| \frac{s^{\kappa+1}}{(\kappa+1)}\phi^{(\kappa+1)}( s) \right| = o\left(\frac 1n\right)
		\end{equation*}
		by the choice of $\kappa$. This in combination with~\eqref{eq:4:KF} yields our claim for $i=0$. The claims for $i=1$ and $i=2$ follow from a direct comparison of $K^{(i)}$ and 
		$K^{(i)}_{\varepsilon_n}$ and an appeal to~\eqref{eq:a:lasttt}.
	\end{proof}

\begin{lemma}\label{lem:a:Hbound}
		Let Assumption~\ref{as:RW} be in force for some $r \leq  2/3$.  Then for $y_n=d_n -   (\log n)^2(d_n^2+nR(d_n)^2)/(d_nR(d_n))$, $h_n = R(d_n)/d_n$ and 
		$\varepsilon_n = \delta d_n$ for $\varepsilon \in (2^{-r}, 1)$ the following inequality holds true for sufficiently large $n$, 
		\begin{equation*}
			 \E \left[ e^{h_n X} \1_{\left\{ X \in \left(\varepsilon_n, y_n\right] \right\}} \right] \leq  \exp \left\{ - (\log n)^2(1+nR(d_n)^2/d_n^2) /\const\right\}.
		\end{equation*}
\end{lemma}
\begin{proof}
 		We will apply the integration by parts formula~\eqref{eq:4:intbyparts}
		with $\psi(s) = e^{h_n s} $, $A = \varepsilon_n$ and $B=y_n$. We have
		\begin{equation}\label{eq:a:estA}
			 \psi(\varepsilon_n)\P\left[X  > \varepsilon_n\right]=  \exp \left(\delta R(d_n) - R(\delta d_n) \right) \leq \exp \left\{ -n/\const \right\},
		\end{equation} 
		since $R(\varepsilon d_n) \sim \varepsilon^rR(d_n) \sim \varepsilon^rn$ and $\varepsilon\in (0,1)$.
		Now note that since $R(x)x^{-d}$ is assumed to be eventually decreasing for some $d\in(0,1)$, we have $R(x)/(R'(x)x)> 1/d>1$ for sufficiently large $x$ and thus
		\begin{equation}\label{eq:a:dif1}
			 h_n y_n-R(y_n) \leq (1-1/d)(\log n)^2(1+ n R(d_n)^2/d_n^2+o(1))
		\end{equation}
		which in turn gives
		\begin{equation*}
			\psi(y_n) \P[X>y_n]\leq \const \exp\left\{- (\log n)^2(1+nR(d_n)^2/d_n^2) /\const  \right\}.
		\end{equation*}
		We finally consider the integral on the r.h.s. of \eqref{eq:4:intbyparts} for which we have
		\begin{align*}
			  \int_{\varepsilon_n }^{y_n} \psi'(s) \P[X>s]\ud s 
			  	& = h_n \int_{\varepsilon_n}^{y_n}  \exp \left\{ h_n s - R(s) \right\} \ud s \nonumber \\
				& \leq\const\cdot  n  \int_{\delta + o(1)}^1\exp \left\{  h_ny_ns-R(y_ns)\right\} \ud s \nonumber\\
								& \leq\const\cdot  n  \int_{\delta + o(1)}^1\exp \left\{  s(h_n y_n-R(y_n))\right\} \ud s,
		\end{align*}
		where in the last line we used the fact that $R(s y_n) \sim s^rR(y_n)$ locally uniformly in $s\in (0, \infty)$.
		We see that if we plug in~\eqref{eq:a:dif1} into the integral and use the fact that $s$ is bounded away from $0$ we will arrive at
		\begin{equation*}
			 \int_{\varepsilon_n}^{y_n} \psi'(s) \P[X>s]\ud s \leq \exp\left\{- (\log n)^2(1+nR(d_n)^2/d_n^2) /\const  \right\}
		\end{equation*}
		which in turn yields our claim.
\end{proof}

\begin{lemma}\label{lem:a:Hbound>23}
		Let Assumption~\ref{as:RW} be in force for $r > 2/3$. Then for $y_n=d_n-nR'(d_n)$ and $h_n = R(d_n)/d_n$ the following inequality holds true 
		\begin{equation*}
			 \E \left[ e^{h_n X} \1_{\left\{ X \in \left(\varepsilon_n, y_n\right] \right\}} \right] \leq  \exp \left\{ \frac{nR(d_n)R'(d_n)}{d_n} \left(\frac{R'(d_n)d_n}{R(d_n)} -1\right)(1+o(1)) \right\}.
		\end{equation*}
\end{lemma}
\begin{proof}
	The argument goes along the exact same lines as for Lemma~\ref{lem:a:Hbound}. Once again apply formula~\eqref{eq:4:intbyparts} with $\psi(s) = e^{h_ns} $, $A = \varepsilon_n$ 
	and $B=y_n$. The estimate~\eqref{eq:a:estA} is 
	still sufficient for our purposes. To estimate the two other terms note that
	\begin{equation*}
		R(y_n) = R(d_n)-nR'(d_n)^2(1+o(1))
	\end{equation*}
	which implies
	\begin{equation*}
		h_ny_n -R(y_n) = \frac{nR(d_n)R'(d_n)}{d_n} \left(\frac{R'(d_n)d_n}{R(d_n)} -1\right)(1+o(1)).
	\end{equation*}
	Using the same arguments as in the previous lemma 
	\begin{multline*}
		\psi\left(y_n\right)  \P\left[X>y_n\right] +   \int_{\varepsilon_n}^{y_n} \psi'(s) \P[X>s]\ud s  \\
		\leq   \exp \left\{ \frac{nR(d_n)R'(d_n)}{d_n} \left(\frac{R'(d_n)d_n}{R(d_n)} -1\right)(1+o(1)) \right\}.
	\end{multline*}
\end{proof}

\bibliographystyle{amsplain}
\bibliography{PPCforBRWwithSET}

\providecommand{\bysame}{\leavevmode\hbox to3em{\hrulefill}\thinspace}
\providecommand{\MR}{\relax\ifhmode\unskip\space\fi MR }
\providecommand{\MRhref}[2]{%
  \href{http://www.ams.org/mathscinet-getitem?mr=#1}{#2}
}
\providecommand{\href}[2]{#2}
\begin{thebibliography}{10}

\bibitem{athreya:1972:branching}
K.~B. Athreya and P.~E. Ney, \emph{Branching processes}, 1972.

\bibitem{13:aidekon:convergence}
E.~Aïdékon, \emph{Convergence in law of the minimum of a branching random
  walk}, Annals of Probability \textbf{41} (2013), 1362--1426.

\bibitem{bhattacharya:2017:point}
A.~Bhattacharya, R.~S. Hazra, and P.~Roy, \emph{Point process convergence for
  branching random walks with regularly varying steps}, Annales de l'institut
  Henri Poincare (B) Probability and Statistics \textbf{53} (2017), 802--818.

\bibitem{76:biggins:first}
J.~D. Biggins, \emph{The first- and last-birth problems for a multitype
  age-dependent branching process}, Advances in Applied Probability (1976).

\bibitem{bingham_goldie_teugels_1987}
N.~H. Bingham, C.~M. Goldie, and J.~L. Teugels, \emph{Regular variation},
  Encyclopedia of Mathematics and its Applications, Cambridge University Press,
  1987.

\bibitem{brunet2011branching}
{\'E}.~Brunet and B.~Derrida, \emph{A branching random walk seen from the tip},
  Journal of Statistical Physics \textbf{143} (2011), no.~3, 420--446.

\bibitem{dembo:1998:large}
A.~Dembo and O.~Zeitouni, \emph{Large deviations techniques and applications},
  1998.

\bibitem{Denisov2008}
D.~Denisov, A.~B. Dieker, and V.~Shneer, \emph{Large deviations for random
  walks under subexponentiality: The big-jump domain}, Annals of Probability
  \textbf{36} (2008), 1946--1991.

\bibitem{Durrett1983}
R.~Durrett, \emph{Maxima of branching random walks}, Zeitschrift für
  Wahrscheinlichkeitstheorie und Verwandte Gebiete \textbf{62} (1983),
  165--170.

\bibitem{dyszewski:max:2022}
P.~Dyszewski, N.~Gantert, and T.~Höfelsauer, \emph{The maximum of a branching
  walk with stretched exponential tails}, to appear in Annales de l’Institut
  Henri Poincaré (B) Probabilités et Statistiques.

\bibitem{Gantert2000}
N.~Gantert, \emph{The maximum of a branching random walk with semiexponential
  increments}, Annals of Probability \textbf{28} (2000), 1219--1229.

\bibitem{J.M.Hammersley1974}
J.~M. Hammersley, \emph{Postulates for subadditive processes}, The Annals of
  Probability \textbf{2} (1974), 652--680.

\bibitem{75:kingman:first}
J.~F.~C. Kingman, \emph{The first birth problem for an age-dependent branching
  process}, The Annals of Probability \textbf{3} (1975), 790--801.

\bibitem{linnik1961limit}
V.~Y. Linnik, \emph{Limit theorems for sums of independent variables taking
  into account large deviations. i}, Theory of Probability \& Its Applications
  \textbf{6} (1961), no.~2, 131--148.

\bibitem{madule2017}
T.~Madaule, \emph{Convergence in law for the branching random walk seen from
  its tip}, Journal of Theoretical Probability \textbf{30} (2017), 27--63.

\bibitem{maillard2013note}
Pascal Maillard, \emph{{A note on stable point processes occurring in branching
  Brownian motion}}, Electronic Communications in Probability \textbf{18}
  (2013), no.~none, 1 -- 9.

\bibitem{1969:nagaev:integral}
A.~V. Nagaev, \emph{Integral limit theorems taking large deviations into
  account when {C}ramér’s condition does not hold. i}, Theory of Probability
  \& Its Applications \textbf{14} (1969), 51--64.

\bibitem{1969:nagaev:integral2}
\bysame, \emph{Integral limit theorems taking large deviations into account
  when {C}ramér’s condition does not hold. ii}, Theory of Probability \& Its
  Applications \textbf{14} (1969), 193--208.

\bibitem{resnick:2013:extreme}
S.~I. Resnick, \emph{{Extreme Values, Regular Variation and Point Processes}},
  Springer New York, New York, NY, 1987.

\bibitem{rozovskii1994probabilities}
L.~V. Rozovskii, \emph{Probabilities of large deviations on the whole axis},
  Theory of Probability \& Its Applications \textbf{38} (1994), no.~1, 53--79.

\bibitem{Shi2015}
Z.~Shi, \emph{Branching random walks}, Springer, 2015.

\bibitem{smith:1995}
P.~J. Smith, \emph{A recursive formulation of the old problem of obtaining
  moments from cumulants and vice versa}, American Statistician \textbf{49}
  (1995), 217--218.

\bibitem{subag2015freezing}
E.~Subag and O.~Zeitouni, \emph{Freezing and decorated {P}oisson point
  processes}, Communications in Mathematical Physics \textbf{337} (2015),
  no.~1, 55--92.

\end{thebibliography}

\end{document}